\documentclass{amsart}

\usepackage{amssymb}
\usepackage{amsthm}
\usepackage{amsmath}
\usepackage{mathrsfs}
\usepackage{amsbsy}
\usepackage{bm}
\usepackage{hyperref}
\usepackage{tikz}
\usepackage{array}
\usepackage[shortlabels]{enumitem}
\usepackage{xcolor}
\usepackage{bbm}
\usepackage{comment}
\usepackage{mathtools}
\usepackage{mathcomp}
\usepackage{xstring}
\usepackage{svg}

\definecolor{Green}{rgb}{0,0.922,0}
\definecolor{DarkGreen}{rgb}{0,0.5,0}
\definecolor{MildGreen}{rgb}{0,0.863,0}
\definecolor{NormalGreen}{rgb}{0,0.8,0}
\definecolor{Pink}{rgb}{1,0,1}
\definecolor{Cyan}{rgb}{0,1,1}
\definecolor{Yellow}{rgb}{1,1,0}
\definecolor{Gold}{rgb}{1,0.73,0}
\definecolor{lavender}{rgb}{0.45,0,1}

\usepackage[noabbrev,capitalise,nameinlink]{cleveref}

\DeclareFontFamily{U}{mathx}{}
\DeclareFontShape{U}{mathx}{m}{n}{<-> mathx10}{}
\DeclareSymbolFont{mathx}{U}{mathx}{m}{n}
\DeclareMathAccent{\widecheck}{0}{mathx}{"71}

\hypersetup{colorlinks=true, citecolor=lavender, linkcolor=lavender,urlcolor=lavender}

\usepackage{hhline}
\allowdisplaybreaks
\usepackage[noadjust]{cite}

\usepackage{caption}
\usepackage[noabbrev,capitalise,nameinlink]{cleveref}
\crefname{conjecture}{Conjecture}{Conjectures}

\newtheorem{theorem}{Theorem}[section]
\newtheorem{proposition}[theorem]{Proposition}

\newtheorem{conjecture}[theorem]{Conjecture}

\newtheorem{lemma}[theorem]{Lemma}

\theoremstyle{definition}
\newtheorem{definition}[theorem]{Definition}
\newtheorem{remark}[theorem]{Remark}

\newtheorem{axiom}[theorem]{Axiom}

\newlist{lemenum}{enumerate}{1}
\setlist[lemenum]{label=(\roman*),ref=Lemma~\thelemma&\roman*}
\crefformat{lemenumi}{\StrBefore{#1}{&}(#2\StrBehind{#1}{&}#3)}
\crefrangeformat{lemenumi}{\StrBefore{#1}{&}(#3\StrBehind{#1}{&}#4--#5\StrAfter{#2}{&}#6)}
\crefmultiformat{lemenumi}{\StrBefore{#1}{&}(#2\StrBehind{#1}{&}#3}{, #2\StrBehind{#1}{&}#3)}{, #2\StrBehind{#1}{&}#3}{, #2\StrBehind{#1}{&}#3)}

\newcommand{\abs}[1]{\left\lvert#1\right\rvert}

\newcommand{\dfn}[1]{\textcolor{blue}{\emph{#1}}}

\newcommand{\N}{\mathbb{N}}
\newcommand{\1}{\mathbbm{1}}

\newcommand{\E}{\mathbb{E}} 
\newcommand{\TT}{T}

\newcommand{\Z}{\mathbb{Z}}
\newcommand{\Q}{\mathbb{Q}}
\newcommand{\R}{\mathbb{R}}
\newcommand{\C}{\mathbb{C}}
\DeclareMathOperator{\interro}{\textnormal{\textinterrobang}}
\DeclareMathOperator{\?}{?}
\DeclareMathOperator{\Cay}{Cay}
\DeclareMathOperator{\idp}{dp_\infty}

\newcommand{\Des}{\mathrm{Des}}

\newcommand{\HH}{\mathbb{H}}
\newcommand{\DD}{\mathbb{D}}

\DeclareMathOperator{\height}{ht}

\begin{document}

\title[]{Reduced Random Walks in the Hyperbolic Plane\textinterrobang}

\author[]{Colin Defant}
\address[]{Department of Mathematics, Harvard University, Cambridge, MA 02138, USA}
\email{colindefant@gmail.com}

\author[]{Mitchell Lee}
\address[]{Department of Mathematics, Harvard University, Cambridge, MA 02138, USA}
\email{mitchell@math.harvard.edu}

\begin{abstract} 
We study Lam's reduced random walk in a hyperbolic triangle group, which we view as a random walk in the upper half-plane. We prove that this walk converges almost surely to a point on the extended real line. We devote special attention to the reduced random walk in $PGL_2(\mathbb{Z})$ (i.e., the $(2,3,\infty)$ triangle group). In this case, we provide an explicit formula for the cumulative distribution function of the limit. This formula is written in terms of the \emph{interrobang function}, a new function $\textnormal{\textinterrobang}\colon[0,1]\to\mathbb{R}$ that shares several of the remarkable analytic and arithmetic properties of Minkowski's question-mark function. 
\end{abstract} 

\maketitle

\section{Introduction}\label{sec:intro}

\begin{figure}
    \begin{center}
        \includegraphics[width=
        \linewidth]{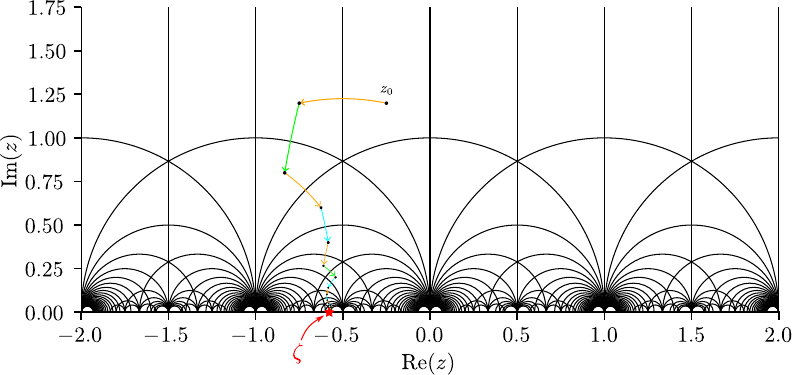}
    \end{center}
    \caption{A trajectory of the reduced random walk in $PGL_2(\Z)$. The limit point $\zeta\in\overline{\R}$ is labeled. }\label{fig:trajectory} 
\end{figure}

\subsection{Random walks on hyperbolic triangle groups}\label{subsec:hyperbolic-reduced-intro}
Let $\Delta$ be a triangle in the hyperbolic plane $\HH$, possibly with ideal vertices. Assume that the angles of $\Delta$ have measures $\pi/p$, $\pi/q$, and $\pi/r$, where $p, q, r \in \{2, 3, 4, \ldots\} \cup \{\infty\}$. Let $s_1, s_2, s_3 \colon \HH \to \HH$ be the reflections in the sides of $\Delta$. The isometries $s_1, s_2, s_3$ generate a hyperbolic group $W$ called the $(p, q, r)$ triangle group, and it is a classical result that the triangles $w(\Delta)$ for $w \in W$ are distinct and form a tiling of the hyperbolic plane \cite[Chapter~II]{MR352287}. We refer to these triangles as \dfn{alcoves}, and we say that two alcoves are \dfn{adjacent} if they share an edge. We will generalize this setup in \cref{sec:preliminaries}.

The \dfn{simple random walk on $W$} starts at the fundamental alcove $\Delta$. At each step, it moves to an alcove chosen uniformly at random from the alcoves adjacent to the current one. Equivalently, the simple random walk on $W$ is the right random walk on $W$ driven by the uniform measure on $\{s_1, s_2, s_3\}$.


Random walks on hyperbolic groups such as $W$ have been studied for at least a century and have continued to attract considerable interest \cite{MR2302729,MR3849626}. For a comprehensive treatment of the subject, see the 2023 book by Lalley \cite{Lalley}.

Minkowski's \emph{question-mark function} $\? \colon [0, 1] \to \R$ has unexpectedly surfaced in the study of random walks on hyperbolic groups. In 2017, Letac and Piccioni studied a particular random walk on the hyperbolic group $PSL_2(\Z)$. They proved that it almost surely converges to a point on the boundary of the hyperbolic plane and found an explicit formula for the cumulative distribution function of the limit in terms of Minkowski's question-mark function \cite{MR3866617}.

In this article, we consider instead the \dfn{(lazy) reduced random walk on $W$}, which is the following variation on the simple random walk on $W$ introduced by Lam in 2015 \cite{Lam}. Again, we start at the fundamental alcove $\Delta$. At each step, we again choose an alcove uniformly at random from the set of alcoves adjacent to the current one. However, in this model, we are not allowed to cross a line more than once. Thus, if moving to the new alcove would cause us to cross a line that we have already crossed, then we stay still at the current alcove; otherwise, we move to the new adjacent alcove.

In this article, we will be primarily concerned with the limiting behavior of the reduced random walk on $W$.

\subsection{Previous research on reduced random walks}

The reduced random walk was originally defined more generally for any Coxeter group $W$. We will describe the generalization in \cref{sec:coxeter}.

For example, suppose that $W$ is the Coxeter group $\mathfrak S_{n+1}$, where the simple reflection $s_i$ ($1 \leq i \leq n$) is the transposition that swaps $i$ and $i+1$. In this case, the transitions in the reduced random walk have a simple combinatorial interpretation in terms of permutations. At each discrete time step, we choose an integer $i\in\{1,\ldots,n\}$ uniformly at random, and we put the entries of the permutation in positions $i$ and $i+1$ into decreasing order (if they are already in decreasing order, we do nothing). This model is a discrete-time analogue of the \emph{oriented swap process}, which was introduced by Angel, Holroyd, and Romik in 2009 \cite{AHR} and which has seen considerable recent interest \cite{BCFGR,BGR,LingfuZhang}. With probability $1$, the reduced random walk eventually reaches the decreasing permutation, which is the unique absorbing state.

Lam studied the reduced random walk when $W$ is an irreducible affine Weyl group \cite{Lam}. In this setting, he proved that the limit of the reduced random walk has a finitely supported probability distribution on the boundary of the Tits cone. Lam found that this distribution is determined by an associated finite-state Markov chain. When $W$ is the affine symmetric group $\widetilde{\mathfrak S}_n$, this finite-state Markov chain coincides with an instance of the \emph{multispecies totally asymmetric simple exclusion process} (\emph{multispecies TASEP}) on a cycle, which had already been studied extensively before Lam's work \cite{EFM,FM,Spitzer} (see also the more recent works \cite{CMW,BL,BL2,BL3}). Lam's seminal article has spawned a great deal of subsequent research \cite{AALP,AL,DefantStoned,LW}.  

\subsection{Results} 

Let us now return to the setting introduced at the beginning of this article, where $W$ is a hyperbolic triangle group. The reduced random walk in $W$ can be viewed geometrically as a random walk $(z_m)_{m\geq 0}$ in the Poincar\'e upper half-plane $\HH$. We will prove (\cref{thm:convergence}) that the sequence $(z_m)_{m\geq 0}$ almost surely converges to a point $\zeta$ on the extended real line $\overline{\R}$. We will also prove (\cref{thm:key,thm:uniqueness}) that the probability distribution of $\zeta$ is uniquely determined by a simple equation.

We will then consider the special case $W = PGL_2(\Z)$, which is a particularly notable example due to its myriad appearances in number theory. (See \cref{fig:trajectory,fig:pgl2}.) We provide an explicit description of the limiting distribution of the reduced random walk via its cumulative distribution function $F_\zeta(x)=\Pr[\zeta\leq x]$ (\cref{thm:pgl2-cumulative}). We find that $F_\zeta$ is continuous and satisfies $\lim_{x\to\infty}F_\zeta(x)=1$, which implies that the limiting distribution of $\zeta$ is atomless. Our description uses an auxiliary function $\interro\colon[0,1]\to\R$ that we call the \emph{interrobang function}. This function, which we initially define on $[0,1]\cap\Q$ via a recurrence relation before extending its domain to all of $[0,1]$, shares many notable analytic and arithmetic properties with Minkowski's question-mark function $\?\colon[0,1]\to\R$ (see \cref{fig:interrobang_plot}). For example, it is strictly increasing and uniformly continuous (\cref{thm:interro-real}), it sends rational numbers to dyadic rational numbers (\cref{prop:rational_to_dyadic}), and it sends quadratic irrational numbers to rational numbers (\cref{thm:quadratic_to_rational}). As a consequence, we find that the cumulative distribution function $F_\zeta$ sends rational numbers to dyadic rational numbers and sends quadratic irrational numbers to rational numbers (\cref{thm:pgl2-cumulative}). 

We also conjecture some additional properties of the interrobang function. On the analytic side, we conjecture that $\interro$ is differentiable and has derivative $0$ almost everywhere (\cref{conj:differentiable}). On the arithmetic side, we conjecture that there exists a dyadic rational number $y \in [0, 3/8]$ such that $\interro^{-1}(y)$ is irrational (\cref{conj:irrational}). In fact, we conjecture that $\interro^{-1}(1/8)$ is transcendental (\cref{conj:transcendental}).

\subsection{Outline} In \cref{sec:coxeter}, we introduce Coxeter groups, the Demazure product, and the reduced random walk. In \cref{sec:preliminaries}, we provide additional background information on hyperbolic triangle groups. In \cref{subsec:oneway}, we prove that the sequence $(z_m)_{m\geq 0}$ geometrically realizing the reduced random walk converges almost surely to a point $\zeta\in\overline{\R}$. In \cref{sec:uniqueness}, we prove that the probability distribution of $\zeta$ is the unique solution to a simple distributional equation. \cref{sec:pgl2} briefly sets the stage for the rest of the paper by discussing how our general setup specializes when $W$ is $PGL_2(\Z)$. In \cref{sec:interro}, we introduce the interrobang function and establish some of its properties. Finally, \cref{sec:pgl2-limit} is devoted to proving an explicit formula for the cumulative distribution function $F_\zeta$ in $PGL_2(\Z)$. 

\section{The reduced random walk on a Coxeter group}\label{sec:coxeter}
For a more complete introduction to the theory of Coxeter groups, see the 2005 textbook of Bj\"orner and Brenti \cite{MR2133266}.

Let $I$ be a finite index set. A \dfn{Coxeter matrix} is a function $m \colon I \times I \to \N \cup \{\infty\}$ such that $m(i, i) = 1$ for all $i \in I$, and $m(i, i') = m(i', i) > 1$ for all distinct $i, i' \in I$. A \dfn{Coxeter system} with the Coxeter matrix $m$ is a pair $(W, S)$, where $W$ is a group, $S = \{s_i \, | \, i \in I\} \subseteq W$ is a generating set, and $W$ has the presentation \[\langle S \, | \, ((s_i s_{i'})^{m(i, i')})_{i, i' \in I}\rangle.\] We refer to $W$ as a \dfn{Coxeter group}, and we refer to the elements of $S$ as \dfn{simple reflections}. It is well known that the simple reflections $s_i$ are distinct and have order $2$.

The \dfn{length} of an element $w\in W$, denoted $\ell(w)$, is the minimum integer $\ell$ such that $w$ can be written in the form $s_{i_1} \cdots s_{i_\ell}$ for some $i_1, \ldots, i_\ell \in I$. For $i \in I$ and $w \in W$, we say that $i$ is a \dfn{left descent} of $w$ if $\ell(s_iw) < \ell(w)$. We say that $i$ is a \dfn{right descent} of $w$ if $\ell(w s_i) < \ell(w)$. The set of left descents of $w$ is denoted by $\Des_L(w)$, and the set of right descents of $w$ is denoted by $\Des_R(w)$.

There is a unique associative binary operation $\star \colon W \times W \to W$ such that for all $i \in I$ and all $w \in W$, we have \[\1 \star w = w \star \1 = w\] and
\begin{equation}\label{eq:demazure-simple-left}
s_i \star w = \begin{cases}
    w & \mbox{if $i \in \Des_L(w)$} \\
    s_i w & \mbox{otherwise}
\end{cases}.
\end{equation}
The same binary operation $\star$ also satisfies the dual equation
\begin{equation}\label{eq:demazure-simple-right}
w \star s_i = \begin{cases}
    w & \mbox{if $i \in \Des_R(w)$} \\
    w s_i & \mbox{otherwise}
\end{cases}.
\end{equation}
We refer to $\star$ as the \dfn{Demazure product}, although it is sometimes called the \dfn{greedy product} or \dfn{$0$-Hecke product}. The Demazure product cannot decrease length; that is, $\ell(u \star w) \geq \max(\ell(u), \ell(w))$ for all $u, w \in W$.

Let $i_1,i_2,i_3,\ldots$ be a sequence of indices chosen independently and uniformly at random from $I$. The \dfn{(lazy) reduced random walk} \cite{Lam} is the sequence of random variables $(u_m)_{m \geq 0}$, where \[u_m = s_{i_1} \star \cdots \star s_{i_m}.\] In other words, the reduced random walk is the (right) random walk on $W$ considered as a monoid with the Demazure product, where the step distribution is uniform on $S$. Thus, the reduced random walk fits into the more general field of random walks on semigroups \cite{FT,GK,HognasMukherjea}. 

By construction, the reduced random walk is a Markov chain on $W$. Unlike the usual simple random walk in $\Cay(W,S)$, the reduced random walk is nondecreasing in length; that is, $\ell(u_m) \leq \ell(u_{m + 1})$ for all $m \geq 0$, and equality holds if and only if $u_m = u_{m + 1}$.

\section{Hyperbolic triangle groups}\label{sec:preliminaries}  
The material in this section can also be found in any reference work on hyperbolic geometry. For a more complete introduction, see the 2005 textbook of Anderson~\cite{MR2161463}.

\subsection{The hyperbolic plane}\label{subsec:hyperbolic}
Let \[\HH = \{z \in \C \, | \, \operatorname{Im}(z) > 0\} \subseteq \C\] denote the upper half-plane. Considering each element $z \in \HH$ as $z = x + iy$, where $x, y \in \R$ and $y > 0$, we may equip $\HH$ with the conformal Riemannian metric \[ds^2 = \frac{dx^2 + dy^2}{y^2}.\] It is well known that with this metric, $\HH$ is a surface with constant negative Gaussian curvature; we refer to $\HH$ as the \dfn{Poincar\'e half-plane model of the hyperbolic plane} or simply the \dfn{hyperbolic plane}. (There are many other ways of constructing the hyperbolic plane, including the \emph{Poincar\'e disk model}, which we will use in \cref{sec:uniqueness}.)

Denote the one-point compactification of $\R$ by $\overline{\R} = \R \cup \{\infty\}$, which is homeomorphic to a circle. 
Denote the one-point compactification of $\HH \cup \R \subseteq \C$ by $\overline{\HH} = \HH \cup \R \cup \{\infty\} = \HH \cup \overline{\R}$, which is homeomorphic to a closed disk. 

A \dfn{(hyperbolic) line} in $\HH$ is a subset of the form \[\{z \in \HH \, | \,\operatorname{Re}(z) = x\}\] or \[\{z \in \HH \, | \, |z - x| = r\},\] where $x, r \in \R$ with $r > 0$. Every geodesic in $\HH$ is contained in a hyperbolic line.

The \dfn{ideal points} of a line $L \subseteq \HH$ are the elements of $\overline{L} \cap \overline{\R}$, where $\overline{L}$ denotes the topological closure of $L$ in $\overline{\HH}$. Every line has two ideal points, and every two distinct elements of $\overline{\R}$ are the ideal points of a unique line.
%

For each hyperbolic line $L \subseteq \HH$, there is a unique anticonformal isometry $t_L \colon \HH \to \HH$ that fixes $L$. If $L = \{z \in \HH \, | \,\operatorname{Re}(z) = x\}$ for some $x \in \R$, then $t_L$ is given by \[t_L(z) = 2x - \overline{z}.\] On the other hand, if $L = \{z \in \HH \, | \, |z - x| = r\}$ for some $x, r \in \R$ with $r > 0$, then $t_L$ is given by \[t_L(z) = \frac{r^2}{\overline{z} - x}+ x.\] We refer to $t_L$ as the \dfn{reflection} in the line $L$; it extends uniquely to a continuous function $t_L \colon \overline{\HH} \to \overline{\HH}$.

\subsection{Hyperbolic triangle groups}\label{subsec:triangle}
Fix lines $L_1, L_2, L_3 \subseteq \HH$ and a point $z_0 \in \HH$. Throughout the rest of this article, we make the following assumption.
\begin{axiom}\label{ax:angles}
For all distinct $i, i' \in \{1, 2, 3\}$, one of the following two conditions holds.
\begin{itemize}
    \item The lines $L_i$ and $L_{i'}$ do not intersect in $\HH$. Moreover, $z_0$ lies between them; that is, $z_0$ is in the connected component of $\HH \setminus (L_i \cup L_{i'})$ whose boundary contains both $L_i$ and $L_{i'}$.
    \item The lines $L_i$ and $L_{i'}$ intersect at an angle with measure $\pi / m(i, i')$, where $m(i, i') \geq 2$ is an integer. This angle contains the point $z_0$ in its interior.
\end{itemize}
\end{axiom}
Define \[\Delta = \{z \in \HH \setminus (L_1 \cup L_2 \cup L_3) \, | \, \text{$z$ and $z_0$ are on the same side of $L_i$ for $i = 1, 2, 3$}\}.\] Clearly, $\Delta$ is an open neighborhood of $z_0$.

For example, \cref{ax:angles} holds if $\Delta \ni z_0$ is the interior of a hyperbolic triangle with sides $L_1$, $L_2$, $L_3$ and angle measures $\pi/m(1, 2)$, $\pi/m(2, 3)$, $\pi/m(1, 3)$. However, \cref{ax:angles} is weak enough to allow for cases where two of the lines $L_1$, $L_2$, $L_3$ do not intersect or even share an ideal point. In this case, the region $\Delta$ bounded by $L_1$, $L_2$, and $L_3$ may be unbounded or even have infinite area; see \cref{fig:non-triangle-triangle}.

\begin{figure}
    \begin{center}
        \includegraphics[width=
        \linewidth]{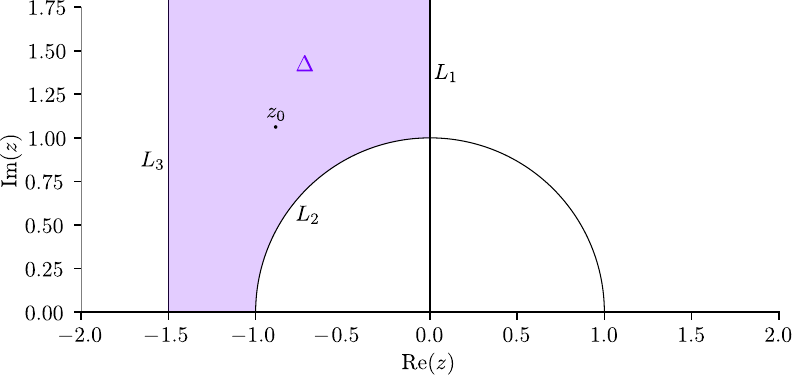}
    \end{center}
    \caption{An example of lines $L_1, L_2, L_3 \subseteq \HH$ and a point $z_0 \in \HH$ satisfying \cref{ax:angles}. In this example, $m(1, 2) = 2$ and $m(2, 3) = m(1, 3) = \infty$.}\label{fig:non-triangle-triangle} 
\end{figure}

The reflections $s_i = t_{L_i}$ generate a group $W$ of isometries of $\HH$ called a \dfn{hyperbolic triangle group}. It is well known that $W$ is infinite, $W$ acts properly on $\HH$, and $(W, S)$ is a Coxeter system with Coxeter matrix $m$, where we take $m(i, i) = 1$ for all $i$ and we take $m(i, i') = \infty$ if the lines $L_i$ and $L_{i'}$ do not intersect \cite[Chapter~II]{MR352287}. The action of $W$ on $\HH$ can also be extended uniquely to a continuous action of $W$ on $\overline{\HH}$.

It is common in the literature to refer to $W$ as the \dfn{$(p, q, r)$ triangle group}, where $p$, $q$, $r$ are the values $m(1, 2)$, $m(2, 3)$, $m(1, 3)$ sorted in nondecreasing order. Note that $W$ is determined up to isomorphism by $p$, $q$, and $r$.

The set $\Delta$ is a fundamental domain for the action of $W$ in the following sense: the sets $w(\Delta)$ for $w \in W$ are disjoint, and their closures cover $\HH$ \cite[Chapter~II]{MR352287}. In other words, the generalized triangles $w(\Delta)$ form a tiling of $\HH$. We refer to the sets $w(\Delta)$ as \dfn{alcoves}. For all $u, w \in W$, the alcoves $u(\Delta)$ and $w(\Delta)$ are adjacent (that is, they share an edge) if and only if $w = us_i$ for some $i \in \{1, 2, 3\}$.

Also associated to the hyperbolic triangle group $W$ is the following arrangement of lines: \[\mathcal{H} = \{L \subseteq \HH \, | \, \text{$L$ is a line and $t_L \in W$}\}.\] Since $t_{wL} = w t_L w^{-1}$ for all lines $L$ and all $w \in W$, the set $\mathcal{H}$ is invariant under the action of $W$. These lines partition the hyperbolic plane into the alcoves $w(\Delta)$.

It is well known that for all $w \in W$, the length $\ell(w)$ is equal to the number of lines $L \in \mathcal{H}$ that separate $z_0$ from $w(z_0)$.

\section{Reduced random walks in the hyperbolic plane}
\label{subsec:oneway}
Use the setup of \cref{subsec:triangle}. We now show how the reduced random walk in the hyperbolic triangle group $W$ can be seen as a random walk $(z_m)_{m \geq 0}$ in the hyperbolic plane $\HH$ itself. In the remainder of the article, we will investigate the limiting behavior of $(z_m)_{m \geq 0}$.

For $i = 1, 2, 3$, define the \dfn{one-way reflection} $\tau_i \colon \overline{\HH} \to \overline{\HH}$ by
\[\tau_i(z) = \begin{cases}
    s_i(z) & \mbox{if $z$ and $z_0$ are on the same side of $L_i$} \\

    z & \mbox{otherwise}
\end{cases}.\]
Each one-way reflection is idempotent: $\tau_i^2 = \tau_i$.

For all $w \in W$, define $\tau_w \colon \overline{\HH} \to \overline{\HH}$ by
\[\tau_w = \tau_{i_1} \circ \cdots \circ \tau_{i_\ell},\] where $i_1 \cdots i_\ell$ is a reduced word for $w$ (that is, $s_{i_1}\cdots s_{i_\ell}=w$ and $w$ has length $\ell$). By Matsumoto's theorem \cite{MR183818}, the element $\tau_w$ is well defined (i.e., does not depend on the choice of reduced word). We have $\tau_w(z_0) = w(z_0)$ and $\tau_{u \star w} = \tau_u \circ \tau_w$ for all $u, w \in W$.

Now, sample elements $i_1, i_2, i_3, \ldots$ uniformly and independently from $\{1, 2, 3\}$. Consider the sequence $(z_m)_{m \geq 0}$, where
\begin{equation}\label{eq:zm-tau}z_m = (\tau_{i_1} \circ \cdots \circ \tau_{i_m})(z_0).\end{equation} By the properties above, we have $z_m = u_m(z_0)$, where \begin{equation}\label{eq:um}u_m = s_{i_1} \star \cdots \star s_{i_m}.\end{equation} Thus, we may think of $(z_m)_{m \geq 0}$ as a geometric interpretation of the reduced random walk. See \cref{fig:pgl2} for an illustration.

The sequence of alcoves \[u_0(\Delta), u_1(\Delta), u_2 (\Delta), \ldots\] that contain the points $z_0, z_1, z_2, \ldots$ respectively was informally described at the end of \cref{subsec:hyperbolic-reduced-intro}. To reiterate, these alcoves form a random walk that never crosses a given line in $\mathcal H$ more than once. This implies the following. 

\begin{lemma}\label{lem:cross}
    For each line $L \in \mathcal H$, there exists an integer $N$ such that the set ${\{z_m \, | \, m \geq N\}}$ lies entirely on one side of the line $L$.
\end{lemma}

We now state and prove our first main theorem about the sequence $(z_m)_{m \geq 0}$.
\begin{theorem}\label{thm:convergence}
    The sequence $(u_m)_{m \geq 0}$ is almost surely not eventually constant. Moreover, if it is not eventually constant, then $(z_m)_{m \geq 0}$ converges to some $\zeta \in \overline{\R}$.
\end{theorem}

\begin{proof}
    By \eqref{eq:um}, we have for all $m$ that $u_{m + 1} = u_{m} \star s_{i_{m + 1}}$, so $u_m = u_{m + 1}$ if and only if $i_{m + 1}$ is a right descent of $u_m$. Because $W$ is infinite, we have ${\Des_R(u_m) \subsetneq \{1, 2, 3\}}$, so $u_{m} = u_{m + 1}$ with probability at most $2/3$. Since $(u_m)_{m \geq 0}$ is a Markov chain, this implies that $(u_m)_{m \geq 0}$ is almost surely not eventually constant. 

    Now, suppose that $(u_m)_{m \geq 0}$ is not eventually constant. For all $m$ with $u_m \neq u_{m + 1}$, we have $\ell(u_{m}) < \ell(u_{m + 1})$. Therefore, \[\lim_{m \to \infty} \ell(u_m) = \infty,\] so every $w \in W$ appears in the sequence $(u_m)_{m \geq 0}$ finitely many times. Since $z_m = u_m(z_0)$ and $z_0$ has trivial stabilizer in $W$, every $z \in \HH$ appears in the sequence $(z_m)_{m \geq 0}$ finitely many times.
    
    Because $\overline{\HH}$ is compact, the sequence $(z_m)_{m\geq 0}$ has a subsequence converging to some $\zeta \in \overline{\HH}$. The action of $W$ on $\HH$ is proper, so $\zeta \in \overline{\R}$. Assume for the sake of contradiction that $(z_m)_{m \geq 0}$ does not converge to $\zeta$. Then $(z_m)_{m\geq 0}$ has a subsequence converging to some $\zeta' \in \overline{\HH}$ with $\zeta' \neq \zeta$. By the same argument as before, $\zeta'\in\overline{\R}$.

    Our strategy will be as follows. We will find a small set $\mathcal{L} \subseteq \mathcal{H}$ of lines that divide the hyperbolic plane $\HH$ into regions $U_1, \ldots, U_k$. By \cref{lem:cross}, there is some $j\in[k]$ such that $z_m\in U_j$ for all sufficiently large $m$. Hence, $\zeta$ and $\zeta'$ are both in the topological closure of $W(z_0) \cap U_j$, where $W(z_0)$ denotes the orbit of $z_0$ under the action of $W$. By choosing the set $\mathcal{L}$ of lines appropriately, we can ensure that for all $j \in [k]$, either the set $W(z_0) \cap U_j$ is finite or the topological closure of $U_j$ does not contain both $\zeta$ and $\zeta'$. This will yield the desired contradiction.
        
    Consider the unique line $M \subseteq \HH$ with ideal points $\zeta$ and $\zeta'$. We consider three cases.

{\bf Case 1.} Suppose that there is a line $L\in\mathcal H$ such that $L \neq M$ and $L \cap M \neq \emptyset$. Then $\zeta$ and $\zeta'$ are on opposite sides of $L$, so we may take $\mathcal{L}=\{L\}$.

{\bf Case 2.} Suppose that $M\in\mathcal H$ and $L \cap M = \emptyset$ for all $L \in \mathcal{H} \setminus \{M\}$. Then, the line $M$ separates two alcoves $v_1(\Delta)$ and $v_2(\Delta)$. We may take $\mathcal{L}$ to be the set \[\{v_1(L_1), v_1(L_2), v_1(L_3), v_2(L_1), v_2(L_2), v_2(L_3)\}\] of all the sides of $v_1(\Delta)$ and $v_2(\Delta)$. (Note that this set has five elements because $M$ is a side of both $v_1(\Delta)$ and $v_2(\Delta)$.)

{\bf Case 3.} Suppose that $L \cap M = \emptyset$ for all $L \in \mathcal{H}$. Then, $M$ is contained in an alcove $v(\Delta)$. We may take $\mathcal{L}$ to be the set \[\{v(L_1), v(L_2), v(L_3)\}\] of the sides of $v(\Delta)$.
\end{proof}

\begin{remark}
    \cref{thm:convergence} implies that the sequence $(z_m)_{m \geq 0}$ converges surely. Contrast this with the reduced random walk in an affine Weyl group, which converges to a point on the boundary of the Tits cone almost surely but not surely \cite[Theorem~1]{Lam}.
\end{remark}

Since the pointwise limit of measurable functions is measurable, the limit $\zeta = \lim_{m \to \infty} z_m$ is a random variable. Our primary goal in the remainder of this article is to determine how $\zeta$ is distributed in $\overline{\R}$. To do so, we will use the following key theorem.

\begin{theorem}\label{thm:key}
    Let $\mu$ denote the probability distribution of the limit \[\zeta = \lim_{m \to \infty} z_m\] in $\overline{\R}$. Then \begin{equation}\label{eq:stationary}\mu = \frac{1}{3}\left((\tau_1)_*\mu + (\tau_2)_*\mu + (\tau_3)_*\mu\right),\end{equation} where $(\tau_i)_*\mu$ denotes the pushforward of $\mu$ by the function $\tau_i$.
\end{theorem}
\begin{proof}
    For $m \geq 0$, let \[z'_m = (\tau_{i_2} \circ \tau_{i_3} \cdots \circ \tau_{i_{m + 1}})(z_0).\] This is the same as the definition of $z_m$ given in \eqref{eq:zm-tau}, but using the sequence of indices $i_2, i_3, i_4, \ldots$ instead of $i_1, i_2, i_3, \ldots$. Therefore, the sequences $(z_m)_{m \geq 0}$ and $(z'_m)_{m \geq 0}$ are identically distributed, so the limit \[\zeta' = \lim_{m \to \infty} z'_m\] also has distribution $\mu$.

    For all $m \geq 0$, we have $\tau_{i_1}(z'_m) = z_{m + 1}$. Since $\tau_{i_1}$ is continuous, we may take the limit as $m \to \infty$ to conclude that \[\tau_{i_1}(\zeta') = \zeta.\] Since $i_1$ is uniformly distributed in $\{1, 2, 3\}$ and independent from $\zeta'$, the theorem follows. 
\end{proof}

\section{Uniqueness of the stationary measure}\label{sec:uniqueness}

In this section, we will prove the following theorem.

\begin{theorem}\label{thm:uniqueness}
    There exists a unique probability measure $\mu$ on $\overline{\R}$ that satisfies the equation
    \[\mu = \frac{1}{3}\left((\tau_1)_*\mu + (\tau_2)_*\mu + (\tau_3)_*\mu\right)\]
    from \cref{thm:key}.
\end{theorem}

In some sense, \cref{thm:key,thm:uniqueness} completely answer the question of how the limit $\zeta = \lim_{m \to \infty} z_m$ is distributed. Later, in \cref{sec:pgl2,sec:interro,sec:pgl2-limit}, we will explicitly describe the distribution of $\zeta$ in the important case $W = PGL_2(\Z)$.
\subsection{The Poincar\'e disk model}\label{subsec:disk}
In this section and only this section, it will be convenient to use a different model for the hyperbolic plane.

The \dfn{Poincar\'e disk model of the hyperbolic plane} is the open unit disk $\DD = \{z \, | \, |z| < 1\} \subseteq \C$ in the complex plane. There is a biholomorphism $\varphi \colon \HH \to \DD$ known as the \dfn{Cayley transform} and given by
\[\varphi(z) = \frac{z - i}{z + i}.\]

Via the Cayley transform, we may transfer all the notions introduced in \cref{subsec:hyperbolic} from $\HH$ to $\DD$. Briefly:
\begin{itemize}
\item There is a conformal Riemannian metric \[ds^2 = \frac{4(dx^2 + dy^2)}{(1 - |z|)^2}\] on $\DD$, which makes $\DD$ into a surface with constant negative Gaussian curvature.
\item The boundary of $\DD$ is the unit circle $S^1 = \{z \in \C \, | \, |z| = 1\}$. Denote the closed unit disk by $\overline{\DD} = \DD \cup S^1$.
\item A \dfn{(hyperbolic) line} in $\DD$ is a subset of the form 
\[\{z \in \DD \, | \, z / c \in \R\}\]
or
\[\{z \in \DD \, | \, |z - c|^2 = |c|^2 - 1\}\]
for some $c \in \C$ with $|c| > 1$. The condition $z / c \in \R$ defines a diameter of the unit circle, whereas the equation $|z - c|^2 = |c|^2 - 1$ defines a circle centered at $c$ that is orthogonal to the unit circle.
\item For each hyperbolic line $L \subseteq \DD$, there is a unique anticonformal isometry $t_L \colon \DD \to \DD$ that fixes $L$. If $L = \{z \in \DD \, | \, z / c \in \R\}$ for some $c \in \C$, then $t_L$ is given by \[t_L(z) = \frac{\overline{z}c}{\overline{c}}.\] In the language of Euclidean geometry, this is the usual reflection in the line $\{z \in \C \, | \, z / c \in \R\}$.

On the other hand, if $\{z \in \DD \, | \, |z - c|^2 = |c|^2 - 1\}$ for some $c \in \C$ with $|c| > 1$, then $t_L$ is given by \[t_L(z) = \frac{|c|^2 - 1}{\overline{z} - \overline{c}} + c.\] In the language of Euclidean geometry, this is the inversion with respect to the circle centered at $c$ with radius $\sqrt{|c|^2 - 1}$.

We refer to $t_L$ as the \dfn{reflection} in the line $L$; it extends uniquely to a continuous function $t_L \colon \overline{\DD} \to \overline{\DD}$.
\end{itemize}
For details, see any reference text on hyperbolic geometry \cite{MR2161463}.

We may also perform the setup of \cref{subsec:triangle} and \cref{subsec:oneway} in the Poincar\'e disk model $\DD$ rather than the Poincar\'e half-plane model $\HH$. This yields a point $z_0 \in \DD$, three lines $L_1, L_2, L_3 \subseteq \DD$, and three one-way reflections $\tau_1, \tau_2, \tau_3 \colon \overline{\DD} \to \overline{\DD}$. We will prove \cref{thm:uniqueness} in this equivalent setting.

By applying a suitable isometry of $\DD$, we may make the additional assumption that $z_0 = 0$. Since $z_0$ does not lie on any of the lines $L_1, L_2, L_3$, this implies that for $i = 1, 2, 3$, the line $L_i$ is of the form \[\{z \in \DD \, | \, |z - c_i|^2 = |c_i|^2 - 1\},\] where $c_i \in \C$ with $|c_i| > 1$. Let us define $r_i = \sqrt{|c_i|^2 - 1}$. The one-way reflection $\tau_i \colon \overline{\DD} \to \overline{\DD}$ is then given by \begin{equation}\label{eq:tau-disk}\tau_i(z) = \begin{cases} \displaystyle\frac{r_i^2}{\overline{z} - \overline{c_i}} + c_i & \mbox{if $|z - c_i| > r_i$} \\ z & \mbox{otherwise}\end{cases}.\end{equation}

\subsection{The contraction properties}
As in the statement of \cref{thm:uniqueness}, let us now consider $\tau_1, \tau_2, \tau_3$ as functions $S^1 \to S^1$ on the boundary of the hyperbolic plane. We will prove that $\tau_1, \tau_2, \tau_3$ satisfy two contraction properties.
\begin{lemma}\label{lem:contract}
    For all $x, y \in S^1$ and all $i = 1, 2, 3$, we have \[|\tau_i(x) - \tau_i(y)| \leq \min\left(1, \frac{r_i}{|x - c_i|}\right)\min\left(1, \frac{r_i}{|y - c_i|}\right)|x - y| \leq |x - y|.\]
\end{lemma}
\begin{proof}
    The second inequality is trivial, so it suffices to prove the first inequality. We have three cases.
    
    {\bf Case 1.} Suppose that $|x - c_i| \leq r_i$ and $|y - c_i| \leq r_i$. By \eqref{eq:tau-disk}, we have $\tau_i(x) = x$ and $\tau_i(y) = y$, so \[|\tau_i(x) - \tau_i(y)| = |x - y| = \min\left(1, \frac{r_i}{|x - c_i|}\right)\min\left(1, \frac{r_i}{|y - c_i|}\right) |x - y|,\]
    and the inequality holds with equality.
    
    {\bf Case 2.} Suppose that $|x - c_i| > r_i$ and $|y - c_i| > r_i$. Then, by \eqref{eq:tau-disk}, we have
    \begin{align*}
        |\tau_i(x) - \tau_i(y)| &= \abs{\left(\frac{r_i^2}{\overline{x} - \overline{c_i}} + c_i\right) - \left(\frac{r_i^2}{\overline{y} - \overline{c_i}} + c_i\right)} \\
        &= \frac{r_i^2\abs{x - y}}{\abs{x - c_i} \abs{y - c_i}} \\
        &= \min\left(1, \frac{r_i}{|x - c_i|}\right)\min\left(1, \frac{r_i}{|y - c_i|}\right) |x - y|,
    \end{align*}
    so the inequality holds with equality again. 

    {\bf Case 3.} Suppose that $|x - c_i| > r_i$ and $|y - c_i| \leq r_i$. (The remaining case that $|x - c_i| \leq r_i$ and $|y - c_i| > r_i$ is analogous.)
    By the intermediate value theorem, there exists $\lambda \in \R$ with $0 < \lambda \leq 1$ such that \[\abs{\lambda x + (1 - \lambda)y - c_i} = r_i.\] Let $z = \lambda x + (1 - \lambda)y$; then, we have $|z - c_i| = r_i$ and $|x - y| = |x - z| + |z - y|$.

    By \eqref{eq:tau-disk} and the triangle inequality, we have
    \begin{align*}
        |\tau_i(x) - \tau_i(y)| &= \abs{\left(\frac{r_i^2}{\overline{x} - \overline{c_i}} + c_i\right) - y} \\
        &= \frac{1}{|x - c_i|} \abs{r_i^2 - (\overline{x} - \overline{c_i}) (y - c_i)} \\
        &= \frac{1}{|x - c_i|} \abs{\left(r_i^2 - (\overline{z} - \overline{c_i}) (y - c_i)\right) - (\overline{x} - \overline{z}) (y - c_i)} \\
        &\leq \frac{1}{|x - c_i|} \left(\abs{r_i^2 - (\overline{z} - \overline{c_i}) (y - c_i)} + |x - z| |y - c_i|\right).
    \end{align*}
    Since $|z - c_i| = r_i$, we may rewrite $\overline{z} - \overline{c_i}$ as $r_i^2/(z - c_i)$ to obtain
    \begin{align*}
        |\tau_i(x) - \tau_i(y)| &\leq \frac{1}{|x - c_i|} \left(\abs{r_i^2 - r_i^2 \frac{y - c_i}{z - c_i}} + |x - z| |y - c_i|\right) \\
        &= \frac{1}{|x - c_i|} \left(r_i^2 \frac{|y - z|}{|z - c_i|} + |x - z| |y - c_i|\right)  \\
        &= \frac{1}{|x - c_i|} \left(r_i |y - z| + |x - z| |y - c_i|\right) \\
        &\leq \frac{r_i}{|x - c_i|}(|y - z| + |x - z|) \\
        &= \frac{r_i |x - y|}{|x - c_i|} \\
        &=\min\left(1, \frac{r_i}{|x - c_i|}\right)\min\left(1, \frac{r_i}{|y - c_i|}\right) |x - y|,
    \end{align*}
    as desired.
\end{proof}
\begin{lemma}\label{lem:strict-contract}
    There exists a real constant $C < 1$, depending only on $L_1, L_2, L_3$, such that for all $x, y \in S^1$, there exists $i \in \{1, 2, 3\}$ with $|\tau_i(x) - \tau_i(y)| < C|x - y|$.
\end{lemma}
\begin{proof}
    Observe that for all $x \in S^1$, there exists $i \in \{1, 2, 3\}$ such that $z_0$ and $x$ lie on the same side of the line $L_i$. That is, there exists $i \in \{1, 2, 3\}$ such that $|x - c_i| > r_i$. Since $S^1$ is compact, there exists a real constant $C < 1$ such that for all $x \in S^1$, there exists $i \in \{1, 2, 3\}$ with $|x - c_i| > r_i / C$. (See \cref{fig:circle}.) 

    Let $x, y \in S^1$, and choose $i \in \{1, 2, 3\}$ such that $|x - c_i| > r_i / C$. By \cref{lem:contract}, we have
    \[|\tau_i(x) - \tau_i(y)| \leq  \min\left(1, \frac{r_i}{|x - c_i|}\right)\min\left(1, \frac{r_i}{|y - c_i|}\right)|x - y| < C|x - y|,\] as desired.
\end{proof}

\begin{figure}
    \begin{center}
        \includegraphics[height=8.7cm]{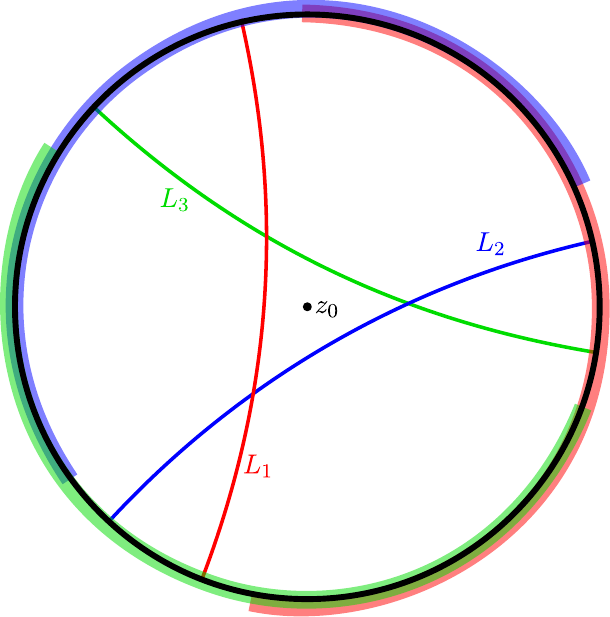}
    \end{center}
    \caption{An illustration of the first part of the proof of \cref{lem:strict-contract}. For $i = 1, 2, 3$, we have drawn the open subset $\{x \in S^1 \, | \, |x - c_i| > r_i / C\} \subseteq S^1$ in {\color{red}red}, {\color{blue}blue}, and {\color{MildGreen}green} respectively. If $C < 1$ is sufficiently close to $1$, then these sets form an open cover of $S^1$.}\label{fig:circle} 
\end{figure}

The following proof is adapted from an argument personally communicated to the authors by Amol Aggarwal. 

\begin{proof}[Proof of \cref{thm:uniqueness}]
    We use the Poincar\'e disk model; see \cref{subsec:disk}. The fact that the measure $\mu$ exists follows directly from \cref{thm:key}. To show uniqueness, let $\mu, \nu$ be two probability measures on $S^1$ such that
    \begin{equation}\label{eq:mu-stationary}\mu = \frac{1}{3}\left((\tau_1)_*\mu + (\tau_2)_*\mu + (\tau_3)_*\mu\right)\end{equation}
    and
    \begin{equation}\label{eq:nu-stationary}\nu = \frac{1}{3}\left((\tau_1)_*\nu + (\tau_2)_*\nu + (\tau_3)_*\nu\right);\end{equation} we will prove that $\mu = \nu$.

    Let $X \sim \mu$ and $Y \sim \nu$ be random variables. We wish to prove that $X$ and $Y$ are identically distributed. It suffices to prove that for every Lipschitz function $f \colon S^1 \to \R$, we have $\E[f(X)] = \E[f(Y)]$ \cite[Volume~II,~Theorem~8.3.2]{MR2267655}.

    Consider the following Markov chain $(X_m, Y_m)_{m \geq 0}$ valued in $S^1 \times S^1$. Choose $i_1, i_2, i_3, \ldots$ independently and uniformly at random from $\{1, 2, 3\}$. Define $X_0 = X$ and $Y_0 = Y$, and for all $m \geq 0$, define $X_{m + 1} = \tau_{i_{m + 1}}(X_m)$ and $Y_{m + 1} = \tau_{i_{m + 1}}(Y_m)$. (Note that $X_m = (\tau_{i_m} \circ \cdots \circ \tau_{i_1})(X_0)$ and $Y_m = (\tau_{i_m} \circ \cdots \circ \tau_{i_1})(Y_0)$; 
    contrast these equations with \eqref{eq:zm-tau}, where the $\tau_i$ are composed in the opposite order.) Observe that by \eqref{eq:mu-stationary} and \eqref{eq:nu-stationary}, we have $X_m \sim \mu$ and $Y_m \sim \nu$ for all $m \geq 0$.
    
    Let $C < 1$ be the constant from \cref{lem:strict-contract}, and let $m \geq 0$. By \cref{lem:contract}, we have $|X_{m + 1} - Y_{m + 1}| \leq |X_m - Y_m|$. Moreover, by \cref{lem:strict-contract}, conditioned on any particular value of $X_m$ and $Y_m$, there is a probability of at least $1/3$ that $|X_{m + 1} - Y_{m + 1}| < C|X_m - Y_m|$. Therefore, 
    \[\E[|X_{m + 1} - Y_{m + 1}|] \leq \frac{C + 2}{3} \E[|X_m - Y_m|].\]
    It follows that \begin{equation}\label{eq:distance-leq}\E[|X_{m} - Y_{m}|] \leq \left(\frac{C + 2}{3}\right)^m |X - Y| \leq 2 \left(\frac{C + 2}{3}\right)^m.\end{equation}
    
    Let $K$ be the Lipschitz constant for the function $f$. We compute
    \begin{align}
        \abs{\E[f(X)] - \E[f(Y)]} &= \lim_{m \to \infty}\abs{\E[f(X_m)] - \E[f(Y_m)]} \label{eq:first}\\
        &\leq \lim_{m \to \infty}\E[\abs{f(X_m) - f(Y_m)}] \label{eq:second}\\
        &\leq K \lim_{m \to \infty}\E[\abs{X_m - Y_m}] \label{eq:third}\\
        &\leq K \lim_{m \to \infty}2 \left(\frac{C + 2}{3}\right)^m \label{eq:fourth} \\
        &= 0 \label{eq:fifth}.
    \end{align}
    In \eqref{eq:first}, we used the fact that for all $m \geq 0$, the random variables $X$ and $X_m$ are identically distributed, and the random variables $Y$ and $Y_m$ are identically distributed. In \eqref{eq:second}, we used the linearity of expectation and Jensen's inequality, which implies that $\abs{\E[Z]} \leq \E[\abs{Z}]$ for any random variable $Z$. In \eqref{eq:third}, we used the fact that $f$ is Lipschitz. In \eqref{eq:fourth}, we directly applied \eqref{eq:distance-leq}. In \eqref{eq:fifth}, we used the fact that $C < 1$. It follows that $\E[f(X)] = \E[f(Y)]$, as desired.
\end{proof}

\section{A worked example: \texorpdfstring{$PGL_2(\Z)$}{PGL₂(ℤ)}}\label{sec:pgl2}
In the remainder of this article, we demonstrate how to solve the equation \eqref{eq:stationary} of \cref{thm:key,thm:uniqueness} in a particular important case. In the setup of \cref{subsec:triangle}, we choose \begin{align*}
    L_1 &= \left\{z \in \HH \, \middle| \, \operatorname{Re}(z) = -\frac{1}{2}\right\} \\
    L_2 &= \left\{z \in \HH \, \middle| \, \abs{z} = 1\right\} \\
    L_3 &= \left\{z \in \HH \, \middle| \, \operatorname{Re}(z) = 0\right\}
\end{align*}
and choose $z_0$ to be an arbitrary complex number with $-1/2 < \operatorname{Re}(z_0) < 0$ and $\abs{z_0} > 1$. Observe that these choices satisfy \cref{ax:angles} with $m(1, 2) = 3$, $m(2, 3) = 2$, and $m(1, 3) = \infty$. The corresponding triangle group $W$ is thus the $(2, 3, \infty)$ triangle group.

The simple generators $s_1$, $s_2$, and $s_3$ of $W$ act on $\HH$ by 
\begin{align*}
s_1(z) &= -1 - \overline{z} \\
s_2(z) &= 1/\overline{z} \\
s_3(z) &= -\overline{z}.
\end{align*}
The $(2, 3, \infty)$ triangle group $W$ is isomorphic to the \dfn{projective general linear group} $PGL_2(\Z)$, which is the quotient of the general linear group $GL_2(\Z)$ by the subgroup $\{I_2, -I_2\}$, where $I_2$ is the $2\times 2$ identity matrix. When there is no risk of confusion, we will use the notation \[\begin{bmatrix}a & b \\ c & d\end{bmatrix}\] to denote both a matrix in $GL_2(\Z)$ and its class in $PGL_2(\Z)$.

The proper action of $PGL_2(\Z)$ on the hyperbolic plane is given by fractional linear transformations; explicitly, for $z\in\HH$, we have \[\begin{bmatrix}a & b \\ c & d\end{bmatrix} (z) = \begin{cases}\displaystyle\frac{az + b}{cz + d} & \mbox{if $ad - bc = 1$} \\\noalign{\vskip5pt}\displaystyle\frac{a\overline{z} + b}{c \overline{z} + d}& \mbox{if $ad - bc = -1$}\end{cases}.\]  This action is conformal if $ad - bc = 1$ and anticonformal if $ad - bc = -1$.

\begin{figure}
    \begin{center}
        \includegraphics[width=
        \linewidth]{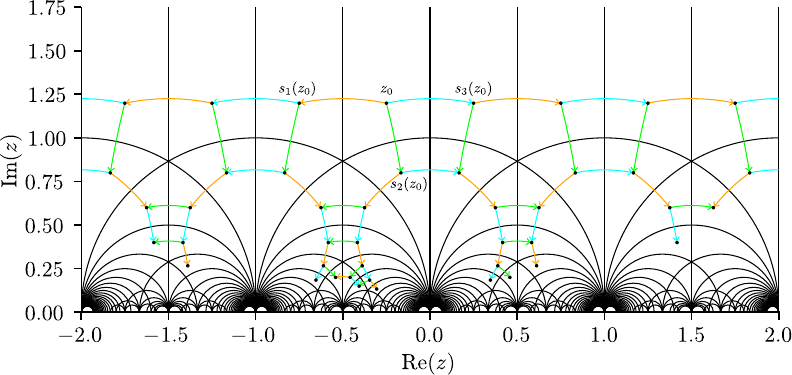}
    \end{center}
    \caption{\label{fig:pgl2}The transition diagram of the reduced random walk in $PGL_2(\Z)$.} 
\end{figure}

\section{The interrobang function}\label{sec:interro}
In this section, we define an auxiliary function $\interro \colon [0,1] \to \R$, called the interrobang function\footnote{The symbol $\interro$ is a ligature of a question mark and an exclamation point called an \emph{interrobang}. It was invented in 1962 by Martin K. Speckter.}, which is related to Minkowski's classical question-mark function $\Q \colon [0, 1] \to \R$. In \cref{thm:pgl2-cumulative}, the interrobang function will allow us to explicitly describe the cumulative distribution function of $\mu$ in the case $W \simeq PGL_2(\Z)$ described in \cref{sec:pgl2}, and thus to provide a complete description of the limiting behavior of the reduced random walk in that case. We encourage the reader to read the statement of \cref{thm:pgl2-cumulative} before proceeding.

Let us introduce some standard notation and definitions. For all $x \in \R$, let $\{x\} = x - \lfloor x \rfloor$ denote the fractional part of $x$. We say that a rational number is \dfn{dyadic} if, when written in lowest terms, its denominator is a power of $2$. A \dfn{quadratic irrational} is an irrational number that is a root of a nonzero quadratic polynomial with rational coefficients.

\begin{definition}
    The \dfn{height function} $\height \colon \Q \to \N$ is given by \[\height\left(\frac{a}{b}\right) = \abs{a} + \abs{b}\] for all integers $a, b$ with $\gcd(a, b) = 1$.
\end{definition}
\begin{remark}
    The function $\height$ is slightly different from the classical height function $h(a / b) = \max\{\abs{a}, \abs{b}\}$ commonly used in arithmetic geometry.
\end{remark}

We also recall one recursive definition of Minkowski's question-mark function on $[0, 1] \cap \Q$.
\begin{definition}[{\cite[Section~4]{MR7929}}]\label{def:question}
    \dfn{Minkowski's question-mark function} $\? : [0, 1] \cap \Q \to \R$ is given by
\[\?(x) = \begin{cases}
    0 & \mbox{if $x = 0$} \\\noalign{\vskip5pt}
    \displaystyle\frac{1}{2^{\lfloor 1 / x \rfloor}}\left(2 -  \?\left(\left\{\frac{1}{x}\right\}\right)\right) & \mbox{otherwise}
\end{cases}.\]
\end{definition}

The interrobang function has a similar, but more complicated, definition. We start by defining the interrobang function on $[0, 1] \cap \Q$; we will later extend it to the whole interval $[0, 1]$ in \cref{thm:interro-real}.

\begin{definition}\label{def:interro}
    The \dfn{interrobang function} $\interro \colon [0, 1] \cap \Q \to \R$ is defined recursively as follows:
    \begin{equation}\label{eq:interro}\interro(x) =
    \begin{cases}
        0 & \mbox{if $x = 0$} \\\noalign{\vskip5pt}
        \displaystyle\frac{1}{4^{\lfloor 1 / x \rfloor}}\left(1 - 2 \interro\left(\left\{\frac{1}{x}\right\}\right)\right) & \mbox{if $0 < x \leq \frac{1}{2}$} \\\noalign{\vskip5pt}
        \displaystyle\frac{3}{8} - \frac{3}{4} \interro\left(\frac{1}{x} - 1\right) - \frac{1}{2} \interro(1 - x) & \mbox{if $\frac{1}{2} < x \leq 1$}
    \end{cases}.
    \end{equation}
    Note that this recursion is well founded by height. That is, for all $x \in [0, 1] \cap \Q$,
    \begin{align*}
        &\text{if $0 < x \leq \frac{1}{2}$, then $\height\left(\left\{\frac{1}{x}\right\}\right) < \height(x)$; and}\\
        &\text{if $\frac{1}{2} < x \leq 1$, then $\height\left(\frac{1}{x} - 1\right) < \height(x)$ and $\height(1 - x) < \height(x)$.}
    \end{align*}
    Thus, \eqref{eq:interro} defines a unique function $\interro \colon [0, 1] \cap \Q \to \R$.
\end{definition}
Minkowski's question-mark function is known to have the following remarkable properties: it is strictly increasing, it can be extended to a continuous function on the real interval $[0, 1]$, it maps rational numbers to dyadic rational numbers, and it maps quadratic irrational numbers to rational numbers \cite[Section~4]{MR7929}. In what follows, we will prove that the interrobang function enjoys these properties as well.

\begin{figure}
    \begin{center}
        \includegraphics[width=
        \linewidth]{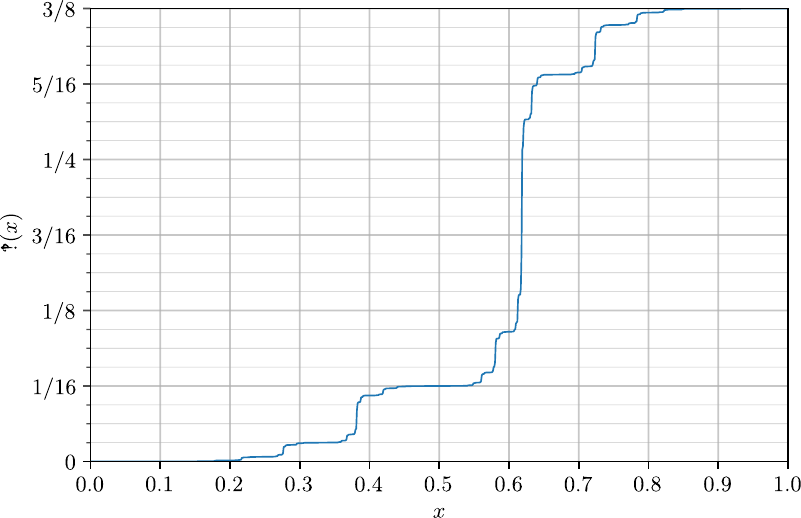}
    \end{center}
    \caption{A plot of the interrobang function.}\label{fig:interrobang_plot}
\end{figure}
\subsection{Basic properties}
\begin{lemma}\label{lem:values}
    We have $\interro(0) = 0$ and $\interro(1) = 3/8$. Moreover, $\interro(1/n) = 1/4^n$ for all $n \geq 2$.
\end{lemma}
\begin{proof}
    This follows directly from \cref{def:interro}.
\end{proof}
\begin{lemma}\label{lem:piecewise}
    Let $n$ be a positive integer, and let $x$ be a rational number satisfying ${1/(n+1) \leq x \leq 1/n}$.
    \begin{enumerate}[(i)]
        \item \label{item:n-gt-1} If $n > 1$, then \[\interro(x) = \frac{1}{4^n}\left(1 - 2 \interro\left(\frac{1}{x} - n\right)\right).\]
        \item \label{item:n-eq-1} If $n = 1$, then \[\interro(x) = \displaystyle\frac{3}{8} - \frac{3}{4} \interro\left(\frac{1}{x} - 1\right) - \frac{1}{2} \interro(1 - x).\]
    \end{enumerate}
\end{lemma}
\begin{proof}
    If $1/(n + 1) < x \leq 1/n$, then this is exactly \eqref{eq:interro}. Otherwise, we have $x = 1/(n + 1)$, and the lemma easily follows from \cref{lem:values}.
\end{proof}

\subsection{Analytic properties}

In this subsection, we prove that $\interro$ extends to a continuous, increasing function on the real interval $[0, 1]$.

\begin{lemma}\label{lem:bounds}
     Let $n$ be a positive integer. For all rational numbers $x$ satisfying ${1/(n + 1) < x < 1/n}$, we have $\interro(1/(n + 1)) < \interro(x) < \interro(1/n)$.
\end{lemma}
\begin{proof}
    We use strong induction on $\height(x)$. If $n > 1$, then by the inductive hypothesis, we have \[0 < \interro\left(\frac{1}{x} - n\right) < \interro(1) = \frac{3}{8}.\] Hence, by \cref{lem:piecewise}, the function value \[\interro(x) = \frac{1}{4^{n}}\left(1 - 2 \interro\left(\frac{1}{x} - n\right)\right)\] satisfies \[\frac{1}{4^n}\left(1 - 2 \cdot \frac{3}{8}\right) < \interro(x) < \frac{1}{4^n} \cdot (1 - 2 \cdot 0).\] By \cref{lem:values}, this simplifies to the desired inequality.

    On the other hand, if $n = 1$, then by the inductive hypothesis, we have \[0 < \interro\left(\frac{1}{x} - 1\right) < \interro(1) = \frac{3}{8}\] and \[0 < \interro(1 - x) < \interro\left(\frac{1}{2}\right) = \frac{1}{16}.\] Hence, by \eqref{eq:interro}, the function value \[\interro(x) = \frac{3}{8} - \frac{3}{4} \interro\left(\frac{1}{x} - 1\right) - \frac{1}{2} \interro(1 - x)\] satisfies
    \[\frac{3}{8} - \frac{3}{4} \cdot \frac{3}{8} - \frac{1}{2} \cdot \frac{1}{16} < \interro(x) < \frac{3}{8} - \frac{3}{4} \cdot 0 - \frac{1}{2} \cdot 0.\]
    By \cref{lem:values}, this simplifies to the desired inequality.
\end{proof}

\begin{proposition}\label{prop:interro-increasing}
    The function $\interro$ is strictly increasing on $[0, 1] \cap \Q$. That is, for all rational numbers $x, y$ with $0 \leq x < y \leq 1$, we have $\interro(x) < \interro(y)$.
\end{proposition}
\begin{proof}
    We use induction on $\height(x) + \height(y)$. If $x = 0$ or if there exists a positive integer $n$ such that $x \leq 1/n \leq y$, then the proposition follows from \cref{lem:bounds}.
    
    Otherwise, there exists a positive integer $n$ such that $1/(n + 1) < x \leq y < 1/n$. If $n > 1$, then we have \[\interro\left(\frac{1}{y} - n\right) < \interro\left(\frac{1}{x} - n\right)\] by the inductive hypothesis. Hence, by \cref{lem:piecewise}, we have
    \[\interro(x) = \frac{1}{4^n} \left(1 - 2\interro\left(\frac{1}{x} - n\right)\right) < \frac{1}{4^n} \left(1 - 2\interro\left(\frac{1}{y} - n\right)\right) = \interro(y),\]
    as desired.

    On the other hand, if $n = 1$, then we have \[\interro\left(\frac{1}{y} - 1\right) < \interro\left(\frac{1}{x} - 1\right)\] and \[\interro\left(1 - y\right) < \interro\left(1 - x\right)\] by the inductive hypothesis. Hence, by \cref{lem:piecewise}, we have
    \[\interro(x) = \frac{3}{8} - \frac{3}{4} \interro\left(\frac{1}{x} - 1\right) - \frac{1}{2} \interro(1 - x) < \frac{3}{8} - \frac{3}{4} \interro\left(\frac{1}{y} - 1\right) - \frac{1}{2} \interro(1 - y) = \interro(y),\]
    as desired.
\end{proof}

Our next goal is to prove that $\interro$ is uniformly continuous on $[0, 1] \cap \Q$. We start by introducing the following definition.

\begin{definition}\label{def:saturated}
    Let $a \leq b$ be real numbers, let $\TT \subseteq \R$, and let $f \colon \TT \to \R$. For $\epsilon > 0$, we say that $f$ is \dfn{$\epsilon$-saturated} on $[a, b]$ if there exist $x_0, \ldots, x_n \in \TT$ such that \[a = x_0 < \cdots < x_n = b\] and $\abs{f(x_{i + 1}) - f(x_{i})} < \epsilon$ for $0 \leq i < n$.
\end{definition}

\begin{lemma}\label{lem:uniformly-continuous-of-saturated}
    Let $a \leq b$ be real numbers, let $\TT \subseteq \R$, and let $f \colon \TT \to \R$. If $f$ is monotone (increasing or decreasing) on $[a, b] \cap \TT$ and $f$ is $\epsilon$-saturated on $[a, b]$ for all $\epsilon > 0$, then $f$ is uniformly continuous on $[a, b] \cap \TT$.
\end{lemma}
\begin{proof}
    Let $\epsilon > 0$. We wish to prove that there exists $\delta > 0$ such that $\abs{f(x) - f(y)} < \epsilon$ for all $x, y \in [a, b] \cap \TT$ satisfying $\abs{x - y} < \delta$. 
    
    By assumption, $f$ is $(\epsilon/2)$-saturated on $[a, b]$, which means that there exist $x_0, \ldots, x_n \in \TT$ such that $a = x_0 < \cdots < x_n = b$ and $\abs{f(x_{i + 1}) - f(x_{i})} < \epsilon / 2$ for $i = 0, \ldots, n - 1$. We may assume, by replacing $\epsilon$ with a smaller constant, that $n > 1$. Let \[\delta = \min\{x_{i + 1} - x_i \, | \, 0 \leq i < n\}.\]
    Then, for all $x, y \in [a, b] \cap \TT$ with $\abs{x - y} < \delta$, there exists $i$ such that $x, y \in [x_i, x_{i+2}]$. Since $f$ is monotone, this implies that 
    \begin{align*}
    \abs{f(x) - f(y)} &\leq \abs{f(x_{i + 2}) - f(x_i)} \\
    &\leq \abs{f(x_{i + 2}) - f(x_{i + 1})} + \abs{f(x_{i + 1}) - f(x_i)} \\
    &< \frac{\epsilon}{2} + \frac{\epsilon}{2} \\
    &= \epsilon,
    \end{align*}
    as desired.
\end{proof}

\begin{lemma}\label{lem:saturated-properties}
    Let $a \leq b \leq c$ be real numbers, let $\TT \subseteq \R$, and let $f, g \colon \TT \to \R$ be functions. Let $\epsilon > 0$, and let $C$ be a real constant.
    \begin{lemenum}
        \item \label{item:saturated-if-abs-lt} If $\abs{f(b) - f(a)} < \epsilon$, then $f$ is $\epsilon$-saturated on $[a, b]$.
        \item \label{item:saturated-upward} If $f$ is $\epsilon$-saturated on $[a, b]$, then $f$ is $\epsilon'$-saturated on $[a, b]$ for all $\epsilon' > \epsilon$.
        \item \label{item:saturated-add-const} If $f$ is $\epsilon$-saturated on $[a, b]$, then $f + C$ is also $\epsilon$-saturated on $[a, b]$.
        \item \label{item:saturated-mul-const} If $f$ is $\epsilon$-saturated on $[a, b]$, then $Cf$ is $(\abs{C}\epsilon)$-saturated on $[a, b]$.
        \item \label{item:saturated-transitive} If $f$ is $\epsilon$-saturated on both $[a, b]$ and $[b, c]$, then $f$ is $\epsilon$-saturated on $[a, c]$.
        \item \label{item:saturated-comp} Let $U \subseteq \R$, and let $h \colon U \to \TT$ be a monotone (increasing or decreasing) function. If $f \circ h$ is $\epsilon$-saturated on $[a, b]$, then $f$ is $\epsilon$-saturated on $[h(a), h(b)]$.
        \item \label{item:saturated-add} Suppose that $f$ and $g$ are each either monotone increasing or monotone decreasing on $[a, b] \cap \TT$. Let $\epsilon, \epsilon' > 0$. If $f$ is $\epsilon$-saturated on $[a, b]$ and $g$ is $\epsilon'$-saturated on $[a, b]$, then $f + g$ is $(\epsilon + \epsilon')$-saturated on $[a, b]$.
    \end{lemenum}
\end{lemma}
\begin{proof}
    We prove \cref{item:saturated-add}; the other parts are trivial. By \cref{def:saturated}, there exist $x_0, \ldots, x_m, y_0, \ldots, y_n$ such that 
    \begin{alignat*}{4}
        a &= x_0 &&< \cdots &&< x_m &&= b \\
        a &= y_0 &&< \cdots &&< y_n &&= b
    \end{alignat*}
    and such that $\abs{f(x_{i + 1}) - f(x_i)} < \epsilon$ for $i = 0, \ldots, m - 1$ and $\abs{g(y_{j + 1}) - g(y_{j})} < \epsilon'$ for $j = 0, \ldots, n - 1$.
    
    Let $z_0, \ldots, z_p$ be the elements of the set $\{x_0, \ldots, x_m, y_0, \ldots, y_n\}$ in increasing order. Then, \[a = z_0 < \cdots < z_p = b.\] Moreover, for all $k$ with $0 \leq k < p$, there exists $i$ such that $z_k, z_{k + 1} \in [x_i, x_{i + 1}]$, and there exists $j$ such that $z_k, z_{k+1} \in [y_j, y_{j + 1}]$. Since $f$ and $g$ are monotone, we conclude that
    \begin{align*}
        \abs{(f + g)(z_{k + 1}) - (f + g)(z_k)} &\leq \abs{f(z_{k+1}) - f(z_k)} + \abs{g(z_{k+1}) - g(z_k)} \\
        &\leq \abs{f(x_{i + 1}) - f(x_i)} + \abs{g(y_{j + 1}) - g(y_j)} \\
        &\leq \epsilon + \epsilon',
    \end{align*}
    so $f + g$ is $(\epsilon + \epsilon')$-saturated on $[a, b]$, as desired.
\end{proof}

\begin{proposition}\label{prop:interro-uniformly-continuous}
    The function $\interro$ is uniformly continuous on $[0, 1] \cap \Q$.
\end{proposition}
\begin{proof}
    We claim that for every nonnegative integer $N$, the function $\interro$ is $((7/8)^N)$-saturated on $[0, 1/2]$ and $(4 \cdot (7/8)^N)$-saturated on $[1/2, 1]$. We prove this claim by induction on $N$. In the base case $N = 0$, the claim follows from \cref{item:saturated-if-abs-lt}.

    For the inductive step, suppose that $\interro$ is $((7/8)^N)$-saturated on $[0, 1/2]$ and ${(4 \cdot (7/8)^N)}$-saturated on $[1/2, 1]$. We wish to prove that $\interro$ is $((7/8)^{N + 1})$-saturated on $[0, 1/2]$ and $(4 \cdot (7/8)^{N + 1})$-saturated on $[1/2, 1]$. Note that the inductive hypothesis implies, by \cref{item:saturated-upward,item:saturated-transitive}, that $\interro$ is $(4 \cdot (7/8)^N)$-saturated on $[0, 1]$.

    First, we prove that $\interro$ is $((7/8)^{N + 1})$-saturated on $[0, 1/2]$. By \cref{item:saturated-transitive}, it suffices to show that $\interro$ is $((7/8)^{N + 1})$-saturated on $[0, 1/(N + 2)]$, as well as on $[1/(n + 1), 1/n]$ for $2 \leq n < N + 2$. For the former, \cref{lem:values} yields \[\abs{\interro\left(\frac{1}{N + 2}\right) - \interro(0)} = \frac{1}{4^{N + 2}} < \left(\frac{7}{8}\right)^{N + 1},\] so $\interro$ is $((7/8)^{N + 1})$-saturated on $[0, 1/(N + 2)]$ by \cref{item:saturated-if-abs-lt}. For the latter, recall from \cref{lem:piecewise} that for all $n \geq 2$ and all $x \in [1/(n + 1), 1/n]$, we have \begin{equation}\label{eq:interro-of-le-half} \interro(x) = \frac{1}{4^n}\left(1 - 2 \interro\left(\frac{1}{x} - n\right)\right).\end{equation}
    Since $\interro$ is $(4 \cdot (7/8)^N)$-saturated on $[0, 1]$, it follows from \cref{item:saturated-comp} that, as a function of $x$, \[ \interro\left(\frac{1}{x} - n\right)\] is $(4 \cdot (7/8)^N)$-saturated on $[1/(n + 1), 1/n]$.
    By \cref{item:saturated-add-const,item:saturated-mul-const} and \eqref{eq:interro-of-le-half}, it follows that $\interro$ is $(2/4^n \cdot 4 \cdot (7 / 8)^{N})$-saturated on $[1/(n + 1), 1 / n]$. Since $n \geq 2$, we have \[\frac{2}{4^n} \cdot 4 \cdot \left(\frac{7}{8}\right)^{N} < \left(\frac{7}{8}\right)^{N + 1},\] so by \cref{item:saturated-upward}, $\interro$ is $((7/8)^{N + 1})$-saturated on $[1/(n + 1), 1/n]$, as desired.

    To complete the induction, it remains to prove that $\interro$ is $(4 \cdot (7/8)^{N + 1})$-saturated on $[1/2, 1]$. Recall that by \cref{lem:piecewise}, we have 
    \begin{equation}\label{eq:interro-of-ge-half}\interro(x) = \displaystyle\frac{3}{8} - \frac{3}{4} \interro\left(\frac{1}{x} - 1\right) - \frac{1}{2} \interro(1 - x)
    \end{equation} for all $x \in [1/2, 1]$.
    Since $\interro$ is $(4 \cdot (7/8)^N)$-saturated on $[0, 1]$, it follows from \cref{item:saturated-comp} that, as a function of $x$, \[\interro\left(\frac{1}{x} - 1\right)\] is $(4 \cdot (7/8)^N)$-saturated on $[1/2, 1]$. Similarly, since $\interro$ is $((7 / 8)^N)$-saturated on $[0, 1/2]$, it follows from \cref{item:saturated-comp} that, as a function of $x$, \[\interro\left(1 - x\right)\] is $((7/8)^N)$-saturated on $[1/2, 1]$.
    
    By \cref{item:saturated-add-const,item:saturated-mul-const,item:saturated-add} and \eqref{eq:interro-of-ge-half}, it follows that $\interro$ is $\epsilon$-saturated on $[0, 1/2]$, where \[\epsilon = \frac{3}{4} \cdot 4 \cdot \left(\frac{7}{8}\right)^N + \frac{1}{2} \cdot \left(\frac{7}{8}\right)^N.\] This equals $4 \cdot (7/8)^{N + 1}$, so the induction is complete, proving the claim.

    Since \[\lim_{N \to \infty} \left(\frac{7}{8}\right)^N = \lim_{N \to \infty} 4 \cdot \left(\frac{7}{8}\right)^N = 0,\] we conclude that $\interro$ is $\epsilon$-saturated on $[0, 1]$ for all $\epsilon > 0$. By \cref{lem:uniformly-continuous-of-saturated}, $\interro$ is uniformly continuous on $[0, 1] \cap \Q$, as desired.
\end{proof}

\begin{theorem}\label{thm:interro-real}
    The function $\interro$ extends uniquely to a continuous, strictly increasing function $\interro \colon [0, 1] \to \R$. Moreover, the functional equation \[\interro(x) =
    \begin{cases}
        0 & \mbox{if $x = 0$} \\\noalign{\vskip5pt}
        \displaystyle\frac{1}{4^{\lfloor 1 / x \rfloor}}\left(1 - 2 \interro\left(\left\{\frac{1}{x}\right\}\right)\right) & \mbox{if $0 < x \leq \frac{1}{2}$} \\\noalign{\vskip5pt}
        \displaystyle\frac{3}{8} - \frac{3}{4} \interro\left(\frac{1}{x} - 1\right) - \frac{1}{2} \interro(1 - x) & \mbox{if $\frac{1}{2} < x \leq 1$}
    \end{cases}\] from \eqref{eq:interro} holds for all $x \in [0, 1]$ (not necessarily rational).
\end{theorem}
\begin{proof}
    This follows from \cref{prop:interro-increasing,prop:interro-uniformly-continuous} and the fact that $[0, 1] \cap \Q$ is dense in $[0, 1]$.
\end{proof}

Dushistova and Moshchevitin proved in 2010 that Minkowski's question-mark function is differentiable almost everywhere and that its derivative is zero almost everywhere \cite{MR2825515}. We conjecture that the interrobang function has this same property.
\begin{conjecture}\label{conj:differentiable} 
    For $x \in [0, 1]$ outside of a set of measure zero, $\interro$ is differentiable at $x$ and $\interro'(x) = 0$.
\end{conjecture}
\subsection{Arithmetic properties}

Like Minkowski's question-mark function, the interrobang function maps rational numbers to dyadic rational numbers.

\begin{proposition}\label{prop:rational_to_dyadic} 
    For all $x \in [0, 1] \cap \Q$, we have that $\interro(x)$ is an explicitly computable dyadic rational number.
\end{proposition}
\begin{proof}
    This follows immediately from \cref{def:interro}.
\end{proof}

Unlike Minkowski's question-mark function, however, it does not appear that $\interro$ maps rational numbers to dyadic rational numbers bijectively.

\begin{conjecture}\label{conj:irrational}
    There exists a dyadic rational number $y \in [0, 3/8]$ such that $\interro^{-1}(y)$ is irrational.
\end{conjecture}
Using a computer search, the authors have verified that neither $\interro^{-1}(1/8) = 0.61242994\ldots$ nor $\interro^{-1}(1/4) = 0.61834758\ldots$ can be written as a fraction $A / B$, where $A$ and $B$ are integers with $0 < A \leq B \leq 100000$. This provides some computational evidence for \cref{conj:irrational}. In fact, we conjecture the following stronger statement.
\begin{conjecture}\label{conj:transcendental}
    The number $\interro^{-1}\left(1/8\right) = 0.61242994\ldots$ is transcendental.
\end{conjecture}

We now prepare to prove that if $x \in [0, 1]$ is a quadratic irrational, then $\interro(x)$ is rational. For this, we will collect some properties of Minkowski's question-mark function and recall the classical Levy--Desplanques theorem.

\begin{lemma}[{\cite[Section~4]{MR7929}}]\label{lem:question}
    Minkowski's question-mark function satisfies the following properties for all $x \in [0, 1]$.
    \begin{lemenum}
        \item\label{item:question-recurrence} The functional equation \[\?(x) = \begin{cases}
    0 & \mbox{if $x = 0$} \\\noalign{\vskip5pt}
    \displaystyle\frac{1}{2^{\lfloor 1 / x \rfloor}}\left(2 -  \?\left(\left\{\frac{1}{x}\right\}\right)\right) & \mbox{otherwise}
\end{cases}\] from \cref{def:question} holds (even if $x$ is irrational).
        \item\label{item:question-rational} If $x$ is a quadratic irrational, then $\?(x)$ is an explicitly computable rational number.
        \item\label{item:question-symm} We have $\?(1 - x) = 1 - \?(x)$.
        \item\label{item:question-bijection} The function $\?$ is strictly increasing and maps $[0, 1]$ to $[0, 1]$ bijectively.
    \end{lemenum}
\end{lemma}
\begin{definition}
    Let $A$ be a square matrix with real entries. We say that $A$ is \dfn{diagonally dominant} if \[\abs{A_{i,i}} > \sum_{j \neq i}\abs{A_{i, j}}\] for all $i$.
\end{definition}
\begin{theorem}[{Levy--Desplanques Theorem, \cite[Corollary~5.6.17]{MR2978290}}]\label{thm:invertible-of-diagonally-dominant}
    If a matrix is diagonally dominant, then it is invertible.
\end{theorem}

\begin{theorem}\label{thm:quadratic_to_rational} 
    Let $x \in [0, 1]$ be a quadratic irrational. Then, $\interro(x)$ is an explicitly computable rational number.
\end{theorem}
\begin{proof}
    By \cref{item:question-rational}, $\?(x)$ is rational. Let $n > 0$ be the denominator of $\?(x)$. For $0 \leq i \leq n$, let $x_i = \?^{-1}(i / n)$, which exists by \cref{item:question-bijection}. We know that $x \in \{x_0, \ldots, x_n\}$. Our strategy will be to construct a system of linear equations with rational coefficients that has $(\interro(x_0), \ldots, \interro(x_{n}))$ as its only solution, thus proving that $\interro(x_0), \ldots, \interro(x_{n})$ are rational and hence that $\interro(x)$ is rational.

    By \cref{item:question-symm}, we have for $0 \leq i \leq n$ that \begin{equation}\label{eq:one-minus-xi} 1 - x_i =x_{n - i}.\end{equation}
    
    By \cref{item:question-recurrence}, we have for $0 < i \leq n$ that \[\frac{i}{n} = \?(x_i) = \frac{1}{2^{\lfloor 1 / x_i\rfloor}}\left(2 - \?\left(\left\{\frac{1}{x_i}\right\}\right)\right).\]
    Rearranging this equation yields
    \[\?\left(\left\{\frac{1}{x_i}\right\}\right) = \frac{2n - 2^{\lfloor 1 / x_i\rfloor} i}{n},\] whence $0 \leq 2n - 2^{\lfloor 1 / x_i\rfloor} i \leq n$ and \begin{equation}\label{eq:frac-inv-xi}\left\{\frac{1}{x_i}\right\} = x_{2n - 2^{\lfloor 1 / x_i\rfloor} i}.\end{equation}

    Let us now use \eqref{eq:one-minus-xi} and \eqref{eq:frac-inv-xi} to find an equation for $\interro(x_i)$. We have three cases.
    \begin{itemize}
        \item If $i = 0$, then $x_i = 0$ and hence \begin{equation}\label{eq:interro-x0}\interro(x_i) = 0.\end{equation}
        \item If $0 < i \leq n / 2$, then by \cref{item:question-bijection}, we have $0 < x_i \leq \?^{-1}(1/2) = 1/2$. Therefore, by \eqref{eq:interro}, we have
        \[\interro(x_i) = \frac{1}{4^{\lfloor1 / x_i\rfloor}} \left(1 - 2 \interro\left(\left\{\frac{1}{x_i}\right\}\right)\right),\]
        which, by \eqref{eq:frac-inv-xi}, simplifies to
        \begin{equation}\label{eq:interro-xi-1}\interro(x_i) = \frac{1}{4^{\lfloor1 / x_i\rfloor}} \left(1 - 2 \interro\left(x_{2n - 2^{\lfloor 1 / x_i\rfloor} i}\right)\right).\end{equation}
        \item If $n / 2 < i \leq n$, then by \cref{item:question-bijection}, we have $1/2 < x_i \leq 1$. Therefore, by \eqref{eq:interro}, we have
        \[\interro(x_i) = \frac{3}{8} - \frac{3}{4} \interro\left(\frac{1}{x_i} - 1\right) - \frac{1}{2} \interro(1 - x_i),\] which, by \eqref{eq:one-minus-xi} and \eqref{eq:frac-inv-xi}, simplifies to
        \begin{equation}\label{eq:interro-xi-2}
            \interro(x_i) = \frac{3}{8} - \frac{3}{4} \interro\left(x_{2n - 2i}\right) - \frac{1}{2} \interro\left(x_{n - i}\right).
        \end{equation}
    \end{itemize}
    Define the matrix $A \in \Q^{(n + 1) \times (n + 1)}$ and the vector $b \in \Q^{n + 1}$ as follows. For $0 \leq i, j \leq n$, let \[A_{i, j} = \begin{cases}
        1 & \mbox{if $i = j$} \\\noalign{\vskip5pt}
        \displaystyle\frac{2}{4^{\lfloor 1 / x_i \rfloor}} & \mbox{if $0 < i \leq \displaystyle\frac{n}{2}$ and $j = 2n - 2^{\lfloor 1 / x_i\rfloor} i$} \\\noalign{\vskip5pt}
        \displaystyle\frac{3}{4} & \mbox{if $\displaystyle\frac{n}{2} < i \leq n$ and $j = 2n - 2i$} \\\noalign{\vskip5pt}
        \displaystyle\frac{1}{2} & \mbox{if $\displaystyle\frac{n}{2} < i \leq n$ and $j = n - i$} \\\noalign{\vskip5pt}
        0 & \mbox{otherwise}
    \end{cases},\]
    and let
    \[b_i = \begin{cases}
        0 & \mbox{if $i = 0$} \\\noalign{\vskip5pt}
        \displaystyle\frac{1}{4^{\lfloor 1 / x_i \rfloor}} & \mbox{if $0 < i \leq \displaystyle\frac{n}{2}$} \\\noalign{\vskip5pt}
        \displaystyle\frac{3}{8} & \mbox{if $\displaystyle\frac{n}{2} < i \leq n$}
    \end{cases}.\]
    Then, the three equations \eqref{eq:interro-x0}, \eqref{eq:interro-xi-1}, and \eqref{eq:interro-xi-2} can be rewritten in the form
    \[A \begin{bmatrix}\interro(x_0) \\ \vdots \\ \interro(x_n)\end{bmatrix} = b.\]
    Let $D$ be the $(n + 1) \times (n + 1)$ diagonal matrix with \[D_{i, i} = \begin{cases} 1 & \mbox{if $0 \leq i \leq \displaystyle \frac{n}{2}$} \\\noalign{\vskip5pt}
    4 & \mbox{if $\displaystyle \frac{n}{2} < i \leq n$}
    \end{cases}.\]
    It can be routinely checked that $AD$ is diagonally dominant. By \cref{thm:invertible-of-diagonally-dominant}, $AD$ is invertible, so $A$ is invertible. It follows that \[\begin{bmatrix}\interro(x_0) \\ \vdots \\ \interro(x_n)\end{bmatrix} = A^{-1}b,\] so $\interro(x_0), \ldots, \interro(x_n)$ are all rational. Since $x \in \{x_0, \ldots, x_n\}$, the result follows.
\end{proof}

\section{The limiting distribution for \texorpdfstring{$PGL_2(\Z)$}{PGL₂(ℤ)}}\label{sec:pgl2-limit}
We are now ready to describe the limiting behavior of the reduced random walk in $PGL_2(\Z)$. A plot of the cumulative distribution function $F_\zeta$ is shown in \cref{fig:cdf_plot}. 

\begin{theorem}\label{thm:pgl2-cumulative}
    Let $W \simeq PGL_2(\Z)$ be the Coxeter group generated by the set $S = \{s_1, s_2, s_3\}$ of simple reflections, where $s_1, s_2, s_3$ are the following transformations of the upper half-plane $\HH$:
    \begin{align*}
    s_1(z) &= -1 - \overline{z} \\
    s_2(z) &= 1/\overline{z} \\
    s_3(z) &= -\overline{z}.
    \end{align*}
    Consider the sequence $(z_m)_{m \geq 0}$, which was defined in \cref{subsec:oneway} as a geometric interpretation of the reduced random walk in $(W, S)$. The limit $\zeta = \lim_{m \to \infty} z_m$ almost surely exists and is a real number. The cumulative distribution function $F_\zeta$ is given by
    \begin{equation}\label{eq:cumulative}F_\zeta(x) = \Pr[\zeta \leq x] = \begin{cases}
        \interro(-1/x) & \mbox{if $x \leq -2$} \\
        1/4 - \interro(-x-1)/2 & \mbox{if $-2 \leq x \leq -1$} \\
        1/4 + \interro(-1-1/x)/2 + \interro(x + 1)/2 & \mbox{if $-1 \leq x \leq -1/2$} \\
        1/2 - \interro(-x)/2 & \mbox{if $-1/2 \leq x \leq 0$} \\
        1/2 + \interro(x) & \mbox{if $0 \leq x \leq 1$} \\
        7/8 + \interro(x-1)/4 & \mbox{if $1 \leq x \leq 2$} \\
        1 - \interro(1/x)/2 & \mbox{if $2 \leq x$}
    \end{cases}\end{equation}
    for all $x \in \R$, where $\interro$ is the interrobang function (\cref{thm:interro-real}). In particular,
    \begin{enumerate}[(i)]
        \item \label{item:continuous-and-strictly-increasing} $F_\zeta$ is continuous and strictly increasing;
        \item \label{item:dyadic-rational} for all $x \in \Q$, we have that $F_\zeta(x)$ is an explicitly computable dyadic rational number;
        \item \label{item:rational} for all quadratic irrationals $x$, we have that $F_\zeta(x)$ is an explicitly computable rational number.
    \end{enumerate}
\end{theorem} 

\begin{proof}
Let $f(x)$ denote the right-hand side of \eqref{eq:cumulative}. It can be routinely checked using \cref{lem:values} that $f$ is well defined at the boundary points $-2, -1, -1/2, 0, 1, 2$. By \cref{thm:interro-real}, $f$ is continuous and strictly increasing on each of the intervals $(-\infty, -2]$, $[-2, -1]$, $[-1, -1/2]$, $[-1/2, 0]$, $[0, 1]$, $[1, 2]$, $[2, \infty)$, so it is continuous and strictly increasing on $\R$. By \cref{prop:rational_to_dyadic}, for all $x \in \Q$, we have that $f(x)$ is an explicitly computable dyadic rational number. By \cref{thm:quadratic_to_rational}, for all quadratic irrationals $x$, we have that $f(x)$ is an explicitly computable rational number.

It remains to prove that $f = F_\zeta$. Since $f$ is continuous and strictly increasing, and since we may easily compute that $\lim_{x \to -\infty} f(x) = 0$ and $\lim_{x \to \infty} f(x) = 1$, there exists a probability measure $\mu'$ on $\R$ with cumulative distribution function $f$. By \cref{thm:key,thm:uniqueness}, it suffices to show that $\mu'$ satisfies the equation
\begin{equation}\label{eq:mu'-stationary}\mu' = \frac{1}{3}\left((\tau_1)_*\mu' + (\tau_2)_*\mu' + (\tau_3)_*\mu'\right),\end{equation} where $\tau_1, \tau_2, \tau_3 \colon \overline{\HH} \to \overline{\HH}$ are the one-way reflections of $W$. Explicitly, for $x \in \R$, we have
\begin{alignat*}{2}
    \tau_1(x) &= \begin{cases} -1 - x & \mbox{if $x > -1/2$} \\ x & \mbox{otherwise}\end{cases} &&= -\frac{1}{2} - \abs{x + \frac{1}{2}}\\
    \tau_2(x) &= \begin{cases} 1/x & \mbox{if $|x| > 1$} \\ x & \mbox{otherwise}\end{cases}\\    \tau_3(x) &= \begin{cases} - x & \mbox{if $x < 0$} \\ x & \mbox{otherwise}\end{cases} &&= \abs{x}.    
\end{alignat*}
Verifying \eqref{eq:mu'-stationary} is routine; the method is to use \cref{lem:piecewise} to show that for all $x \in \Q$, both sides of \eqref{eq:mu'-stationary} assign the same measure to the set $(-\infty, x]$. The computation involves casework depending on where $x$ falls among finitely many intervals. We illustrate the case $x \in [-3, -2]$ and omit the details for the other cases.

Suppose that $x \in [-3, -2]$. We have \begin{align*}
    &\phantom{{}={}}\frac{1}{3}\left((\tau_1)_*\mu' + (\tau_2)_*\mu' + (\tau_3)_*\mu'\right)((-\infty,x]) \\ 
    &=    \frac{1}{3}\left(\mu'(\tau_1^{-1}((-\infty,x])) + \mu'(\tau_2^{-1}((-\infty,x]))+ \mu'(\tau_3^{-1}((-\infty,x]))\right) \\ 
    &=\frac{1}{3}\left(\mu'([-1-x,\infty)\cup(-\infty,x]) + \mu'(\emptyset)+ \mu'(\emptyset)\right) \\ 
    &=\frac{1}{3}\left(1-f(-1-x)+f(x)\right) \\ 
    &=\frac{1}{3}\left(1-\left(\frac{7}{8}+\frac{1}{4}\interro(-2-x)\right)+\interro\left(-\frac{1}{x}\right)\right) \\ 
    &=\frac{1}{24}-\frac{1}{12}\interro(-2-x)+\frac{1}{3}\interro\left(-\frac{1}{x}\right).     
\end{align*}
Since $1/3\leq -1/x\leq 1/2$, we can apply \cref{lem:piecewise} with $n=2$ to find that 
\begin{align*}
    &\phantom{{}={}}\frac{1}{24}-\frac{1}{12}\interro(-2-x)+\frac{1}{3}\interro\left(-\frac{1}{x}\right)\\
    &=\frac{1}{24}-\frac{1}{12}\interro(-2-x)+\frac{1}{3}\cdot\frac{1}{16}\left(1-2\interro\left(-2-x\right)\right) \\ 
    &= \frac{1}{16}\left(1-2\interro\left(-2-x\right)\right) \\
    &= \interro\left(-\frac{1}{x}\right) \\
    &= f(x) \\
    &= \mu'((-\infty, x]),
\end{align*}
as desired.
\end{proof}

\begin{figure}
    \begin{center}
        \includegraphics[width=
        \linewidth]{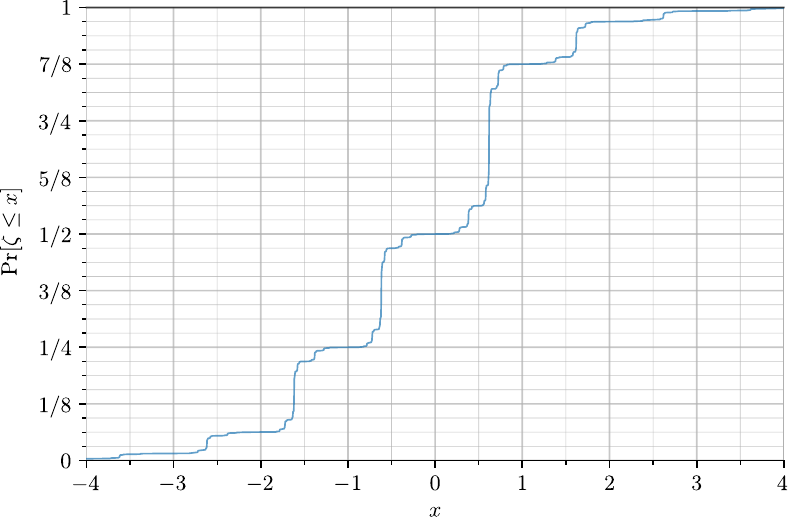}
    \end{center}
    \caption{A plot of the cumulative distribution function $F_\zeta$.}\label{fig:cdf_plot} 
\end{figure}

\section*{Acknowledgments}
The authors thank Amol Aggarwal for communicating part of the argument that we used to prove \cref{thm:uniqueness}. 
Colin Defant was supported by the National Science Foundation under Award No.\ 2201907 and by a Benjamin Peirce Fellowship at Harvard University. 

\bibliographystyle{plain}
\bibliography{main}

\begin{thebibliography}{10}

\bibitem{AALP}
Erik Aas, Arvind Ayyer, Svante Linusson, and Samu Potka.
\newblock Limiting directions for random walks in classical affine {W}eyl groups.
\newblock {\em Int. Math. Res. Not. IMRN}, (4):3092--3137, 2023.

\bibitem{MR2161463}
James~W. Anderson.
\newblock {\em Hyperbolic geometry}.
\newblock Springer Undergraduate Mathematics Series. Springer-Verlag London, Ltd., London, second edition, 2005.

\bibitem{AHR}
Omer Angel, Alexander Holroyd, and Dan Romik.
\newblock The oriented swap process.
\newblock {\em Ann. Probab.}, 37(5):1970--1998, 2009.

\bibitem{AL}
Arvind Ayyer and Svante Linusson.
\newblock Correlations in the multispecies {TASEP} and a conjecture by {L}am.
\newblock {\em Trans. Amer. Math. Soc.}, 369(2):1097--1125, 2017.

\bibitem{BL}
Jinho Baik and Zhipeng Liu.
\newblock T{ASEP} on a ring in sub-relaxation time scale.
\newblock {\em J. Stat. Phys.}, 165(6):1051--1085, 2016.

\bibitem{BL3}
Jinho Baik and Zhipeng Liu.
\newblock Fluctuations of {TASEP} on a ring in relaxation time scale.
\newblock {\em Comm. Pure Appl. Math.}, 71(4):747--813, 2018.

\bibitem{BL2}
Jinho Baik and Zhipeng Liu.
\newblock Multipoint distribution of periodic {TASEP}.
\newblock {\em J. Amer. Math. Soc.}, 32(3):609--674, 2019.

\bibitem{BCFGR}
Elia Bisi, Fabio~Deelan Cunden, Shane Gibbons, and Dan Romik.
\newblock The oriented swap process and last passage percolation.
\newblock {\em Random Structures Algorithms}, 60(4):690--715, 2022.

\bibitem{MR2133266}
Anders Bj\"{o}rner and Francesco Brenti.
\newblock {\em Combinatorics of {C}oxeter groups}, volume 231 of {\em Graduate Texts in Mathematics}.
\newblock Springer, New York, 2005.

\bibitem{MR2267655}
V.~I. Bogachev.
\newblock {\em Measure theory. {V}ol. {I}, {II}}.
\newblock Springer-Verlag, Berlin, 2007.

\bibitem{BGR}
Alexey Bufetov, Vadim Gorin, and Dan Romik.
\newblock Absorbing time asymptotics in the oriented swap process.
\newblock {\em Ann. Appl. Probab.}, 32(2):753--763, 2022.

\bibitem{CMW}
Sylvie Corteel, Olya Mandelshtam, and Lauren Williams.
\newblock From multiline queues to {M}acdonald polynomials via the exclusion process.
\newblock {\em Amer. J. Math.}, 144(2):395--436, 2022.

\bibitem{DefantStoned}
Colin Defant.
\newblock Random combinatorial billiards and stoned exclusion processes, ar{X}iv:2406.07858, 2024.

\bibitem{MR2825515}
A.~A. Dushistova and N.~G. Moshchevitin.
\newblock On the derivative of the {M}inkowski {$?(x)$} function.
\newblock {\em Fundam. Prikl. Mat.}, 16(6):33--44, 2010.

\bibitem{EFM}
Martin~R. Evans, Pablo~A. Ferrari, and Kirone Mallick.
\newblock Matrix representation of the stationary measure for the multispecies {TASEP}.
\newblock {\em J. Stat. Phys.}, 135(2):217--239, 2009.

\bibitem{FM}
Pablo~A. Ferrari and James~B. Martin.
\newblock Stationary distributions of multi-type totally asymmetric exclusion processes.
\newblock {\em Ann. Probab.}, 35(3):807--832, 2007.

\bibitem{FT}
Behrang Forghani and Giulio Tiozzo.
\newblock Random walks of infinite moment on free semigroups.
\newblock {\em Probab. Theory Related Fields}, 175(3-4):1099--1122, 2019.

\bibitem{GK}
Robert~D. Gray and Mark Kambites.
\newblock On cogrowth, amenability, and the spectral radius of a random walk on a semigroup.
\newblock {\em Int. Math. Res. Not. IMRN}, (12):3753--3793, 2020.

\bibitem{HognasMukherjea}
G\"oran H\"ogn\"as and Arunava Mukherjea.
\newblock {\em Probability measures on semigroups}.
\newblock Probability and its Applications (New York). Springer, New York, second edition, 2011.
\newblock Convolution products, random walks, and random matrices.

\bibitem{MR2978290}
Roger~A. Horn and Charles~R. Johnson.
\newblock {\em Matrix analysis}.
\newblock Cambridge University Press, Cambridge, second edition, 2013.

\bibitem{Lalley}
Steven~P. Lalley.
\newblock {\em Random walks on infinite groups}, volume 297 of {\em Graduate Texts in Mathematics}.
\newblock Springer, Cham, [2023] \copyright 2023.

\bibitem{Lam}
Thomas Lam.
\newblock The shape of a random affine {W}eyl group element and random core partitions.
\newblock {\em Ann. Probab.}, 43(4):1643--1662, 2015.

\bibitem{LW}
Thomas Lam and Lauren Williams.
\newblock A {M}arkov chain on the symmetric group that is {S}chubert positive?
\newblock {\em Exp. Math.}, 21(2):189--192, 2012.

\bibitem{MR3866617}
G\'erard Letac and Mauro Piccioni.
\newblock Random walks in the hyperbolic plane and the {M}inkowski question mark function.
\newblock {\em J. Theoret. Probab.}, 31(4):2376--2389, 2018.

\bibitem{MR352287}
Wilhelm Magnus.
\newblock {\em Noneuclidean tesselations and their groups}, volume Vol. 61 of {\em Pure and Applied Mathematics}.
\newblock Academic Press [Harcourt Brace Jovanovich, Publishers], New York-London, 1974.

\bibitem{MR3849626}
Joseph Maher and Giulio Tiozzo.
\newblock Random walks on weakly hyperbolic groups.
\newblock {\em J. Reine Angew. Math.}, 742:187--239, 2018.

\bibitem{MR2302729}
Jean Mairesse and Fr\'ed\'eric Math\'eus.
\newblock Random walks on free products of cyclic groups.
\newblock {\em J. Lond. Math. Soc. (2)}, 75(1):47--66, 2007.

\bibitem{MR183818}
Hideya Matsumoto.
\newblock G\'en\'erateurs et relations des groupes de {W}eyl g\'en\'eralis\'es.
\newblock {\em C. R. Acad. Sci. Paris}, 258:3419--3422, 1964.

\bibitem{MR7929}
R.~Salem.
\newblock On some singular monotonic functions which are strictly increasing.
\newblock {\em Trans. Amer. Math. Soc.}, 53:427--439, 1943.

\bibitem{Spitzer}
Frank Spitzer.
\newblock Interaction of {M}arkov processes.
\newblock {\em Advances in Math.}, 5:246--290, 1970.

\bibitem{LingfuZhang}
Lingfu Zhang.
\newblock Shift-invariance of the colored {TASEP} and finishing times of the oriented swap process.
\newblock {\em Adv. Math.}, 415:Paper No. 108884, 60, 2023.

\end{thebibliography}
\end{document}